\documentclass{article}
\usepackage[utf8]{inputenc}
\usepackage{bm,amsfonts,amsthm,amsmath,amssymb}
\usepackage{graphicx,color}
\usepackage{mathrsfs}
\usepackage{indentfirst}
\usepackage{bbm}
\usepackage{appendix}

\DeclareMathAlphabet{\mathpzc}{OT1}{pzc}{m}{it}

\def\eqdefa{\buildrel\hbox{\footnotesize def}\over =}
\newcommand{\ve}{\varepsilon}
\newcommand{\ud}{\mathrm{d}}

\newcommand{\vv}{\mathbf{v}}

\newcommand{\aaa}{\mathbf{a}}
\newcommand{\xx}{\mathbf{x}}
\newcommand{\yy}{\mathbf{y}}
\newcommand{\nn}{\mathbf{n}}

\newcommand{\hh}{\mathbf{h}}

\newcommand{\sss}{\mathbf{s}}

\newcommand{\A}{\mathbf{A}}

\newcommand{\FF}{\mathbf{F}}

\newcommand{\CB}{\mathcal{B}}

\newcommand{\CC}{\mathcal{C}}
\newcommand{\CS}{\mathcal{S}}

\newcommand{\CV}{\mathcal{V}}
\newcommand{\CE}{\mathcal{E}}

\newcommand{\CF}{\mathcal{F}}

\newcommand{\CG}{\mathcal{G}}
\newcommand{\CU}{\mathcal{U}}
\newcommand{\CT}{\mathcal{T}}

\newcommand{\CA}{\mathcal{A}}
\newcommand{\CP}{\mathcal{P}}
\newcommand{\CH}{\mathcal{H}}

\newcommand{\CW}{\mathcal{W}}

\newcommand{\ML}{\mathscr{L}}

\newcommand{\MF}{\mathscr{F}}

\newcommand{\Fi}{\mathfrak{i}}
\newcommand{\Fp}{\mathfrak{p}}

\newcommand{\BR}{{\mathbb{R}^2}}

\newcommand{\BOm}{\mathbf{\Omega}}

\newtheorem{theorem}{Theorem}

\newtheorem{lemma}[theorem]{Lemma}

\numberwithin{theorem}{section}
\numberwithin{equation}{section}
\newtheorem{definition}[theorem]{Definition}

\allowdisplaybreaks[4]
\title{Uniqueness of global weak solutions to the frame hydrodynamics for biaxial nematic phases in $\mathbb{R}^2$}
\author{
Sirui Li\footnote{ School of Mathematics and Statistics, Guizhou University, Guiyang 550025, China (srli@gzu.edu.cn) },
Chenchen Wang\footnote{ School of Mathematics and Statistics, Guizhou University, Guiyang 550025, China (ccwangmath@163.com) },
Jie Xu\footnote{LSEC and NCMIS, Institute of Computational Mathematics and Scientific/Engineering Computing (ICMSEC), Academy of Mathematics and Systems Science (AMSS), Chinese Academy of Sciences, Beijing, China (xujie@lsec.cc.ac.cn)}}

\date{}

\begin{document}
\maketitle
\begin{abstract}
We consider the hydrodynamics for biaxial nematic phases described by a field of orthonormal frame, which can be derived from a molecular-theory-based tensor model. We prove the uniqueness of global weak solutions to the Cauchy problem of the frame hydrodynamics in dimensional two. The proof is mainly based on the suitable weaker energy estimates within the Littlewood--Paley analysis. We take full advantage of the estimates of nonlinear terms with rotational derivatives on $SO(3)$, together with cancellation relations and dissipative structures of the biaxial frame system.


\end{abstract}


\section{Introduction}

This paper is devoted to the uniqueness of weak solutions to the two-dimensional hydrodynamics of biaxial nematic liquid crystals. Since its local anisotropy is non-axisymmetric, the orientational order needs to be represented by an orthonormal frame field $\Fp\in SO(3)$ instead of a unit vector field for the uniaxial nematics. The biaxial hydrodynamics is a coupled system between evolution of the frame field and the Navier--Stokes equation.


The uniaxial hydrodynamics, i.e. the well-known Ericksen--Leslie model \cite{E-61,E-91,Les}, has been studied extensively \cite{Lin4,WZZ4,Ball}.
On the analytical aspect, the existence and uniqueness of global weak solutions are established \cite{Lin2, Hong,HX,HLW,WW,Lin3,Lin5,WWZ-zju,LTX}; the well-posedness of smooth solutions has also been studied, for the original model \cite{WZZ1,WW,HNPS} and an inertial analogue \cite{JL,CW}.
More results are summarized in several review articles \cite{Lin4,WZZ4,Ball}.

For the biaxial hydrodynamics \cite{S-A,Liu-M,BP,S-W-M,GV2,LLC}, its form has been written using various variables, which should be equivalent.
A recent work \cite{LX} derive its coefficients from a molecular-theory-based two-tensor model \cite{XZ} based on the Hilbert expansion, so that they are expressed by molecular parameters.
In the derivation, the energy dissipation is maintained, and the Ericksen--Leslie model is recovered if the local anisotropy is uniaxial.
The model in \cite{LX} can be formulated by all the components of the orthonormal frame field, which turns out to be convenient for analyses.
This formulation is utilized to establish the well-posedness of smooth solutions in $\mathbb{R}^d(d=2,3)$ and the global existence of weak solutions in $\mathbb{R}^2$ \cite{LWX}, which, to our knowledge, is the first analytic work for the \emph{full-form} biaxial hydrodynamics, although an artificial simplified model has been discussed \cite{LLWa}.

The aim of this paper is to prove the uniqueness of the global weak solution established in \cite{LWX}.
To this end, it requires to derive closed energy estimates for the system formed by the difference of two solutions.
This is, however, difficult to be done in the natural energy space because of some difficult nonlinear terms in the biaxial hydrodynamics.
We shall apply the Littlewood--Paley theory to consider weaker metrics (see (\ref{Phi-t-metric}) and (\ref{Wj-xt})), which
has been succesfully carried out for the Ericksen--Leslie model \cite{LTX,WWZ-zju,Xu-Zhangzf,DW}.
A major difficulty lies within the handling of the nonlinear rotational derivatives on $SO(3)$ in the energy dissipative law.
To facilitate the analysis, it is necessary to rewrite the the orientational elasticity in a equivalent form (see (\ref{new-elasitic-density})).
Moreover, to obtain the dissipative estimates of higher-order derivative terms (see Lemma \ref{key-lemma}), a key ingredient is to employ the decomposition by the tangential space at a point on $SO(3)$ and its orthogonal complement. Such a decomposition has played a key role in the preceding work \cite{LWX}.

Before we present the main result, we prepare some notations for orthogonal frames and tensors that will be repeatedly used throughout our analysis, followed by writing down the biaxial frame hydrodynamics.

\subsection{Preliminaries}\label{Preliminaries}

We denote by $\Fp=(\nn_1,\nn_2,\nn_3)\in SO(3)$ the orthonormal frame, which is constituted by three mutually perpendicular unit vectors.  The symbol $\otimes$ stands for the tensor product. For any two tensors $U$ and $V$ with the same order, the dot product $U\cdot V$ is defined by summing up the product of the corresponding coordinates, i.e.,
\begin{align*}
  U\cdot V=U_{i_1\cdots i_n}V_{i_1\cdots i_n},\quad|U|^2=U\cdot U.
\end{align*}
Hereafter, the Einstein summation convention on repeated indices is assumed.

To express symmetric tensors conveniently, the monomial notation will be adopted as follows:
\begin{align*}
  \nn^{k_1}_1\nn^{k_2}_2\nn^{k_3}_3=\Big(\underbrace{\nn_1\otimes\cdots\otimes\nn_1}_{k_1}\otimes\underbrace{\nn_2\otimes\cdots\otimes\nn_2}_{k_2}\otimes\underbrace{\nn_3\otimes\cdots\otimes\nn_3}_{k_3}\Big)_{\rm sym}.
\end{align*}
In other words, when the symbol $\otimes$ is omitted in a product, it implies that the resulting tensor has been symmetrized. For example, for the frame $\Fp=(\nn_1,\nn_2,\nn_3)$ and $\alpha,\beta=1,2,3,$ we have
\begin{align*}
  &\nn_{\alpha}^2=\nn_{\alpha}\otimes\nn_{\alpha},\quad\nn_{\alpha}\nn_{\beta}=\frac{1}{2}(\nn_{\alpha}\otimes\nn_{\beta}+\nn_{\beta}\otimes\nn_{\alpha}),\quad\alpha\neq\beta.
\end{align*}
In this way, the $3\times3$ identity tensor $\Fi$ can be expressed as a polynomial, i.e.,
$\Fi=\nn_1^2+\nn_2^2+\nn_3^3$.

The differential operators on $SO(3)$ will be involved when describing the frame hydrodynamics. For any frame $\Fp=(\nn_1,\nn_2,\nn_3)\in SO(3)$, we denote by $T_{\Fp}SO(3)$ the tangential space of $SO(3)$ at a point $\Fp$, which can be spanned by the orthogonal basis:
\begin{align*}
  V_1=(0,\nn_3,-\nn_2),\quad V_2=(-\nn_3,0,\nn_1),\quad V_3=(\nn_2,-\nn_1,0).
\end{align*}
Then, its orthogonal complement space $(T_{\Fp}SO(3))^{\perp}$ can be spanned by
\begin{align*}
  W_1=&(0,\nn_3,\nn_2),\quad W_2=(\nn_3,0,\nn_1),\quad W_3=(\nn_2,\nn_1,0),\\
  W_4=&(\nn_1,0,0),\quad
  W_5=(0,\nn_2,0),\quad
  W_6=(0,0,\nn_3).
\end{align*}
Consequently, we define the differential operators $\ML_k(k=1,2,3)$ on $T_{\Fp}SO(3)$ by taking the inner products of the orthogonal basis $\{V_1,V_2,V_3\}$ and
$\partial/\partial\Fp=(\partial/\partial\nn_1,\partial/\partial\nn_2,\partial/\partial\nn_3)$, that is,
\begin{align}\label{diff-ML1-3}
  \left\{
  \begin{aligned}
    &\ML_1\eqdefa V_1\cdot\frac{\partial}{\partial\Fp}=\nn_3\cdot\frac{\partial}{\partial\nn_2}-\nn_2\cdot\frac{\partial}{\partial\nn_3},\\
    &\ML_2\eqdefa V_2\cdot\frac{\partial}{\partial\Fp}=\nn_1\cdot\frac{\partial}{\partial\nn_3}-\nn_3\cdot\frac{\partial}{\partial\nn_1},\\
    &\ML_3\eqdefa V_3\cdot\frac{\partial}{\partial\Fp}=\nn_2\cdot\frac{\partial}{\partial\nn_1}-\nn_1\cdot\frac{\partial}{\partial\nn_2},
  \end{aligned}
  \right.
\end{align}
where $\ML_k(k=1,2,3)$ are actually the derivatives along the infinitesimal rotation about $\nn_k$. Acting the operators $\ML_k(k=1,2,3)$ on $\nn_l$, it follows that $\ML_k\nn_l=\epsilon^{klp}\nn_p$ with $\epsilon^{klp}$ being the Levi-Civita symbol.
The operators $\ML_k$ are also suitable for a functional if we replace $\partial/\partial\Fp$ by the variational derivative $\delta/\delta\Fp$.

In order to deal with the estimate of the higher-order derivative terms with the differential operators $\ML_k(k=1,2,3)$, we will resort to the orthogonal decomposition with respect to the tangential space $T_{\Fp}SO(3)$. More specifically, for any two matrices $A,B\in\mathbb{R}^{3\times3}$, the inner product $A\cdot B$ can be expressed as
\begin{align}\label{Orth-decomp-SO3}
A\cdot B=\sum_{k=1}^3\frac{1}{|V_k|^2}(A\cdot V_k)(B\cdot V_k)+\sum_{k=1}^6\frac{1}{|W_k|^2}(A\cdot W_k)(B\cdot W_k).
\end{align}
On the other hand, for any frame $\Fp=(\nn_1,\nn_2,\nn_3)\in SO(3)$, any first-order differential operator $\mathcal{D}$, and $\alpha,\beta=1,2,3$, it holds that
\begin{align}\label{import-properties}
\left\{
\begin{aligned}
&\mathcal{D}\nn_1=(\mathcal{D}\nn_1\cdot\nn_2)\nn_2+(\mathcal{D}\nn_1\cdot\nn_3)\nn_3\\
&\mathcal{D}\nn_2=(\mathcal{D}\nn_2\cdot\nn_1)\nn_1+(\mathcal{D}\nn_2\cdot\nn_3)\nn_3\\
&\mathcal{D}\nn_3=(\mathcal{D}\nn_3\cdot\nn_2)\nn_2+(\mathcal{D}\nn_3\cdot\nn_1)\nn_1\\
&\mathcal{D}\nn_{\alpha}\cdot\nn_{\beta}+\mathcal{D}\nn_{\beta}\cdot\nn_{\alpha}=\mathcal{D}(\nn_{\alpha}\cdot\nn_{\beta})=0,\\
&W_{\alpha}\cdot\mathcal{D}\Fp=0,
\end{aligned}
    \right.
\end{align}
where we reiterate that $\{W_{\alpha}\}^6_{\alpha=1}$ is the orthogonal basis of $(T_{\Fp}SO(3))^{\perp}$.

\subsection{Frame hydrodynamics}\label{frame-hydrodynamics}

The local orientation of biaxial nematic phases is described by an orthonormal frame $\Fp=(\nn_1,\nn_2,\nn_3)\in SO(3)$.
The corresponding orientational elasticity, can be written as
\begin{align}\label{elastic-energy}
\CF_{Bi}[\Fp]=\int_{\mathbb{R}^2}f_{Bi}(\Fp,\nabla\Fp)\ud\xx,
\end{align}
where the deformation free energy density $f_{Bi}$ has the following form \cite{SV,GV1,Xu2}:
\begin{align*}
f_{Bi}(\Fp,\nabla\Fp)=&\,\frac{1}{2}\Big(K_{1}(\nabla\cdot\nn_{1})^{2}+K_{2}(\nabla\cdot\nn_{2})^{2}+K_{3}(\nabla\cdot\nn_{3})^{2}\\
&\,+K_{4}(\nn_{1}\cdot\nabla\times\nn_{1})^{2}+K_{5}(\nn_{2}\cdot\nabla\times\nn_{2})^{2}+K_{6}(\nn_{3}\cdot\nabla\times\nn_{3})^{2}\\
&\,+K_{7}(\nn_{3}\cdot\nabla\times\nn_{1})^{2}+K_{8}(\nn_{1}\cdot\nabla\times\nn_{2})^{2}+K_{9}(\nn_{2}\cdot\nabla\times\nn_{3})^{2}\\
&\,+K_{10}(\nn_{2}\cdot\nabla\times\nn_{1})^2+K_{11}(\nn_{3}\cdot\nabla\times\nn_{2})^2+K_{12}(\nn_{1}\cdot\nabla\times\nn_{3})^2\\
&\,+\gamma_{1}\nabla\cdot[(\nn_{1}\cdot\nabla)\nn_{1}-(\nabla\cdot\nn_{1})\nn_{1}]+\gamma_{2}\nabla\cdot[(\nn_{2}\cdot\nabla)\nn_{2}-(\nabla\cdot\nn_{2})\nn_{2}]\\
&\,+\gamma_{3}\nabla\cdot[(\nn_{3}\cdot\nabla)\nn_{3}-(\nabla\cdot\nn_{3})\nn_{3}]\Big).
\end{align*}
The elastic energy density $f_{Bi}$ is composed of twelve bulk terms and three surface terms.
The coefficients $K_i(i=1,\cdots,12)$ of bulk terms are all positive.
Each surface term is a null Lagrangian, so that the coefficients $\gamma_i>0 (i=1,2,3)$ will be determined later as needed.
We remark that the above form should be the most convenient one for our analysis later, although other equivalent forms are available.

In order to formulate the frame hydrodynamics, we introduce a set of local basis formed by nine second-order tensors, that is, the identity tensor $\Fi$, five symmetric traceless tensors,
\begin{align*}
    \sss_1=\nn^2_1-\frac13\Fi,\quad \sss_2=\nn^2_2-\nn^2_3,\quad \sss_3=\nn_1\nn_2,\quad \sss_4=\nn_1\nn_3,\quad \sss_5=\nn_2\nn_3,
\end{align*}
and three asymmetric traceless tensors,
\begin{align*}
    \aaa_1=\nn_1\otimes\nn_2-\nn_2\otimes\nn_1,\quad \aaa_2=\nn_3\otimes\nn_1-\nn_1\otimes\nn_3,\quad \aaa_3=\nn_2\otimes\nn_3-\nn_3\otimes\nn_2.
\end{align*}

The frame hydrodynamics for biaxial nematic phases consists of evolution equations for the orthonormal frame field $\Fp=(\nn_1,\nn_2,\nn_3)\in SO(3)$, coupled with the Navier--Stokes equations for the fluid velocity field $\vv$. These equations are given by (see \cite{LX} for details):
\begin{align}
&\,\chi_1\dot{\nn}_2\cdot\nn_3-\frac{1}{2}\chi_1\BOm\cdot\aaa_3-\eta_1\A\cdot\sss_5+\ML_1\CF_{Bi}=0,\label{frame-equation-n1}\\
&\,\chi_2\dot{\nn}_3\cdot\nn_1-\frac{1}{2}\chi_2\BOm\cdot\aaa_2-\eta_2\A\cdot\sss_4+\ML_2\CF_{Bi}=0,\label{frame-equation-n2}\\
&\,\chi_3\dot{\nn}_1\cdot\nn_2-\frac{1}{2}\chi_3\BOm\cdot\aaa_1-\eta_3\A\cdot\sss_3+\ML_3\CF_{Bi}=0,\label{frame-equation-n3}\\
&\,\Fp=(\nn_1,\nn_2,\nn_3)\in SO(3),\label{frame-SO3}\\
&\,\dot{\vv}=-\nabla p+\eta\Delta\vv+\nabla\cdot\sigma+\mathfrak{F},\label{yuan-frame-equation-v}\\
&\,\nabla\cdot\vv=0,\label{yuan-imcompressible-v}
\end{align}
where we use the dot derivative to denote the material derivative $\partial_t+\vv\cdot\nabla$.
The constraint (\ref{frame-SO3}) is necessary in this formulation since other equations does not guarantee it.
In the equation of $\vv=(v_1,v_2,v_3)^T$, the pressure $p$ ensures the incompressibility (\ref{yuan-imcompressible-v}), and $\eta$ is the viscous coefficient.
The divergence of the stress $\sigma$ should be comprehended as $(\nabla\cdot\sigma)_i=\partial_j\sigma_{ij}$.
Now let us specify $\sigma$, for which we introduce $\A$ and $\BOm$ to represent the symmetric and skew-symmetric components of the velocity gradient $\kappa_{ij}=\partial_jv_i$, respectively, i.e.,
\begin{align*}
\A=\frac{1}{2}(\kappa+\kappa^T),\quad\BOm=\frac{1}{2}(\kappa-\kappa^T).
\end{align*}
Then, the stress $\sigma=\sigma(\Fp,\vv)$ is given by
\begin{align}\label{sigma-e}
\sigma(\Fp,\vv)=&\,\beta_1(\A\cdot\sss_1)\sss_1+\beta_0(\A\cdot\sss_2)\sss_1+\beta_0(\A\cdot\sss_1)\sss_2+\beta_2(\A\cdot\sss_2)\sss_2\nonumber\\
&\,+\beta_3(\A\cdot\sss_3)\sss_3-\eta_3\Big(\dot{\nn}_1\cdot\nn_2-\frac{1}{2}\BOm\cdot\aaa_1\Big)\sss_3\nonumber\\
&\,+\beta_4(\A\cdot\sss_4)\sss_4-\eta_2\Big(\dot{\nn}_3\cdot\nn_1-\frac{1}{2}\BOm\cdot\aaa_2\Big)\sss_4\nonumber\\
&\,+\beta_5(\A\cdot\sss_5)\sss_5-\eta_1\Big(\dot{\nn}_2\cdot\nn_3-\frac{1}{2}\BOm\cdot\aaa_3\Big)\sss_5\nonumber\\
&\,+\frac{1}{2}\eta_3(\A\cdot\sss_3)\aaa_1-\frac{1}{2}\chi_3\Big(\dot{\nn}_1\cdot\nn_2-\frac{1}{2}\BOm\cdot\aaa_1\Big)\aaa_1\nonumber\\
&\,+\frac{1}{2}\eta_2(\A\cdot\sss_4)\aaa_2-\frac{1}{2}\chi_2\Big(\dot{\nn}_3\cdot\nn_1-\frac{1}{2}\BOm\cdot\aaa_2\Big)\aaa_2\nonumber\\
&\,+\frac{1}{2}\eta_1(\A\cdot\sss_5)\aaa_3-\frac{1}{2}\chi_1\Big(\dot{\nn}_2\cdot\nn_3-\frac{1}{2}\BOm\cdot\aaa_3\Big)\aaa_3,
\end{align}
where the coefficients in (\ref{sigma-e}) are derived from molecular parameters and satisfy the following conditions (see {\rm \cite{LX}} for details):
\begin{align}\label{coefficient-conditions}
\left\{
\begin{array}{l}
\beta_i\geq0,~i=1,\cdots,5,\quad \chi_j>0,~j=1,2,3,\quad \eta>0,\vspace{1ex}\\
\beta^2_0\leq\beta_1\beta_2,~~\eta^2_1\leq\beta_5\chi_1,~~\eta^2_2\leq\beta_4\chi_2,~~\eta^2_3\leq\beta_3\chi_3.
\end{array}
  \right.
\end{align}
These conditions are crucial for the energy dissipation law to hold.
The body force $\mathfrak{F}$ is defined by
\begin{align}\label{external-force-F}
\mathfrak{F}_i=\partial_i\nn_1\cdot\nn_2\ML_3\CF_{Bi}+\partial_i\nn_3\cdot\nn_1\ML_2\CF_{Bi}+\partial_i\nn_2\cdot\nn_3\ML_1\CF_{Bi}.
\end{align}

It is preferred that the orthonormal constraint (\ref{frame-SO3}) is maintained automatically.
For this purpose, using (\ref{import-properties}), we reformulate the equations of $\Fp$ by all its coordinates:
\begin{align}
\dot{\nn}_1=&\Big(\frac{1}{2}\BOm\cdot\aaa_1+\frac{\eta_3}{\chi_3}\A\cdot\sss_3-\frac{1}{\chi_3}\ML_3\CF_{Bi}\Big)\nn_2\nonumber\\
&-\Big(\frac{1}{2}\BOm\cdot\aaa_2+\frac{\eta_2}{\chi_2}\A\cdot\sss_4-\frac{1}{\chi_2}\ML_2\CF_{Bi}\Big)\nn_3,\label{new-frame-equation-n1}\\
\dot{\nn}_2=&-\Big(\frac{1}{2}\BOm\cdot\aaa_1+\frac{\eta_3}{\chi_3}\A\cdot\sss_3-\frac{1}{\chi_3}\ML_3\CF_{Bi}\Big)\nn_1\nonumber\\
&+\Big(\frac{1}{2}\BOm\cdot\aaa_3+\frac{\eta_1}{\chi_1}\A\cdot\sss_5-\frac{1}{\chi_1}\ML_1\CF_{Bi}\Big)\nn_3,\label{new-frame-equation-n2}\\
\dot{\nn}_3=&\Big(\frac{1}{2}\BOm\cdot\aaa_2+\frac{\eta_2}{\chi_2}\A\cdot\sss_4-\frac{1}{\chi_2}\ML_2\CF_{Bi}\Big)\nn_1\nonumber\\
&-\Big(\frac{1}{2}\BOm\cdot\aaa_3+\frac{\eta_1}{\chi_1}\A\cdot\sss_5-\frac{1}{\chi_1}\ML_1\CF_{Bi}\Big)\nn_2,\label{new-frame-equation-n3}\\
\dot{\vv}=&-\nabla p+\eta\Delta\vv+\nabla\cdot(\sigma+\sigma^d),\label{frame-equation-v}\\
\nabla\cdot\vv=&0,\label{imcompressible-v}
\end{align}
where the condition that $\Fp=(\nn_1,\nn_2,\nn_3)\in SO(3)$ is implied by the equations themselves.
We have also rewrite the body force:
\begin{align*}
(\mathfrak{F})_i=\partial_j\sigma^d_{ij}+\partial_i\widetilde{p},\quad \sigma^d_{ij}=-\frac{\partial f_{Bi}}{\partial(\partial_j\Fp)}\cdot\partial_i\Fp,
\end{align*}
with $\widetilde{p}$ absorbed into the pressure term $p$ (see Lemma \ref{body-force-lemma} for details).

The relations between coefficients (\ref{coefficient-conditions}) will guarantee that the biaxial hydrodynamics (\ref{new-frame-equation-n1})--(\ref{imcompressible-v}) has the following basic energy dissipation law \cite{LWX}:
\begin{align}\label{energy-law}
&\frac{\ud}{\ud t}\Big(\frac{1}{2}\int|\vv|^2\ud\xx+\CF_{Bi}[\Fp]\Big)
=-\eta\|\nabla\vv\|^2_{L^2}-\sum^3_{k=1}\frac{1}{\chi_k}\|\ML_k\CF_{Bi}\|^2_{L^2}\nonumber\\
&\qquad-\bigg(\beta_1\|\A\cdot\sss_1\|^2_{L^2}+2\beta_0\int(\A\cdot\sss_1)(\A\cdot\sss_2)\ud\xx+\beta_2\|\A\cdot\sss_2\|^2_{L^2}\bigg)\nonumber\\
&\qquad-\Big(\beta_3-\frac{\eta^2_3}{\chi_3}\Big)\|\A\cdot\sss_3\|^2_{L^2}
-\Big(\beta_4-\frac{\eta^2_2}{\chi_2}\Big)\|\A\cdot\sss_4\|^2_{L^2}
-\Big(\beta_5-\frac{\eta^2_1}{\chi_1}\Big)\|\A\cdot\sss_5\|^2_{L^2}.
\end{align}

To simplify the presentation, compared with the original model derived by \cite{LX}, we have assumed that the concentration of rigid molecules, together with the product of the Boltzmann constant and the absolute temperature, are all equal to one.

\subsection{The main result}

For any given constant orthonormal frame $\Fp^{*}=(\nn^{*}_1,\nn^{*}_2,\nn^{*}_3)\in SO(3)$, we denote
\begin{align*}
    H^1_{\Fp^{*}}\big(\mathbb{R}^2,SO(3)\big)\eqdefa\big\{\Fp=(\nn_1,\nn_2,\nn_3): \Fp-\Fp^{*}\in H^1(\mathbb{R}^2;\mathbb{R}^3),~|\nn_i|=1~{\rm a.e.~ in}~\mathbb{R}^2, i=1,2,3\big\}.
\end{align*}
Given two constants $\tau$ and $T$ with $0\leq\tau<T$, two spaces $V(\tau,T)$ and $H(\tau,T)$ are defined by
\begin{align*}
V(\tau,T)\eqdefa&\bigg\{\Fp=(\nn_1,\nn_2,\nn_3):\mathbb{R}^2\times[\tau,T]\rightarrow SO(3)\Big|\Fp(t)\in H^1_{\Fp^*}\big(\mathbb{R}^2,SO(3)\big)~\text{for a.e.}~ t\in[\tau,T]\\
&\quad\text{and satisfies}~\mathop{\mathrm{ess\sup}}_{\tau\leq t\leq T}\int_{\mathbb{R}^2}|\nabla\Fp(\cdot,t)|^2\ud\xx+\int^T_{\tau}\int_{\mathbb{R}^2}(|\nabla^2\Fp|^2+|\partial_t\Fp|^2)\ud\xx\ud t<\infty,\\
&\quad\text{where}~|\nabla\Fp(\cdot,t)|^2=\sum^3_{i=1}|\nabla\nn_i(\cdot,t)|^2,~|\nabla^2\Fp|^2=\sum^3_{i=1}|\nabla^2\nn_i|^2,~|\partial_t\Fp|^2=\sum^3_{i=1}|\partial_t\nn_i|^2\bigg\},\\
H(\tau,T)\eqdefa&\bigg\{\vv:\mathbb{R}^2\times[\tau,T]\rightarrow\mathbb{R}^2\Big| \vv~\text{is measurable and satisfies}\\
&\quad \mathop{\mathrm{ess\sup}}_{\tau\leq t\leq T}\int_{\mathbb{R}^2}|\vv(\cdot,t)|^2\ud\xx
+\int^T_{\tau}\int_{\mathbb{R}^2}|\nabla\vv|^2\ud\xx\ud t<\infty\bigg\}.
\end{align*}

Let us introduce the global existence result of weak solutions of the frame hydrodynamics for biaxial nematics in dimensional two, which is established in \cite{LWX}.
\begin{theorem}[see \cite{LWX}]\label{global-posedness-theorem}
Let $(\Fp_0, \vv_0)\in H^1_{\Fp^{*}}\big(\mathbb{R}^2,SO(3)\big)\times L^2(\mathbb{R}^2,\mathbb{R}^2)$ be given initial data with $\nabla\cdot\vv_0=0$ and $\Fp_0=\big(\nn_1(\xx,0),\nn_2(\xx,0),\nn_3(\xx,0)\big)\in SO(3)$. Then there exists a global weak solution $(\Fp,\vv):\mathbb{R}^2\times[0,+\infty)\rightarrow SO(3)\times\mathbb{R}^2$ of the biaxial frame system {\rm (\ref{new-frame-equation-n1})--(\ref{imcompressible-v})} such that the solution $(\Fp,\vv)$ is smooth in $\mathbb{R}^2\times((0,+\infty)\setminus\{T_{l}\}^L_{l=1})$ for a finite number of times $\{T_{l}\}^L_{l=1}$. Furthermore, there exists two constants $\ve_0>0$ and $R_0>0$ such that each singular point $(x^l_i,T_l)$ is characterized by the condition
\begin{align*}
    \limsup_{t\nearrow T_l}\int_{B_R(x^l_i)}\big(|\nabla\Fp|^2+|\vv|^2\big)(\cdot,t)\ud t>\ve_0,\quad |\nabla\Fp|^2=\sum^3_{i=1}|\nabla\nn_i|^2,
\end{align*}
for any $R>0$ with $R\leq R_0$.
\end{theorem}
The solution constructed in Theorem \ref{global-posedness-theorem} is also called the Struwe-type weak solution.
In the following, we will show the uniqueness of the weak solution in the class $V(0,T)\times H(0,T)$ for arbitrary $T>0$.
\begin{theorem}\label{uniqueness-theorem}
  Assume that the assumptions in Theorem \ref{global-posedness-theorem} are satisfied. Let $(\Fp^{(1)},\vv^{(1)})$ and $(\Fp^{(2)},\vv^{(2)})$ be two weak solutions of the frame hydrodynamic system \eqref{new-frame-equation-n1}--\eqref{imcompressible-v}
  determined by Theorem \ref{global-posedness-theorem},
  subject to the same initial data $(\Fp_0,\vv_0)$. Then we have $(\Fp^{(1)},\vv^{(1)})=(\Fp^{(2)}(t),\vv^{(2)}(t))$ for any $t\in[0,+\infty)$.
\end{theorem}

To show the above theorem, we introduce a suitable weaker metrics to arrive at a closed energy estimate.
The main obstacle is to control the higher-order derivative terms involving $\CH^{\Delta_j}_k(k=1,2,3)$ (see (\ref{H-Delta-j}) for the definitions).
We overcome it by utilizing the orthogonal decomposition (\ref{Orth-decomp-SO3}) to obtain the following dissipative estimate (see Lemma \ref{key-lemma}):
\begin{align*}
\sum^3_{k=1}\frac{1}{\chi_k}\|\CH^{\Delta_j}_k\|^2_{L^2}\geq \frac{2\gamma^2}{\chi}\|\Delta\Delta_j\delta_{\Fp}\|^2_{L^2}+\text{lower~order~terms},
\end{align*}
where $\gamma=\min\{\gamma_1,\gamma_2,\gamma_3\}$, $\chi=\max\{\chi_1,\chi_2,\chi_3\}$, $\delta_{\Fp}=\Fp^{(1)}-\Fp^{(2)}$, and $\Delta_j$ is the Littlewood--Paley operator defined in Subsection \ref{Little-Paley-subsect}.

The rest of the paper is organized as follows. In Section 2, we briefly introduce the Littlewood--Paley theory and some useful lemmas to be utilized subsequently.
Section 3 is devoted to the proof of the uniqueness of weak solution to the frame hydrodynamics. The dissipative estimate of higher-order derivative terms, which ensures the closure of energy estimates, will be discussed.

\section{Littlewood--Paley theory and some useful lemmas}

In this section, we introduce the Littlewood-Paley theory (see \cite{BCD} for more details) and some useful lemmas on nonlinear and commutator estimates.
Algebraic structures of $\frac{\delta\CF_{Bi}}{\delta\Fp}$ are also discussed.

\subsection{Littlewood-Paley theory}\label{Little-Paley-subsect}

We denote by $\CS(\mathbb{R}^d)$ the Schwartz space. Let $\CC$ be the annulus $\{\xi\in\mathbb{R}^d:\frac{3}{4}\leq|\xi|\leq\frac{8}{3}\}$ and $\CB$ the ball $\{\xi\in\mathbb{R}^d: |\xi|\leq\frac{4}{3}\}$. There exist two nonnegative radial functions $\chi, \varphi\in\CS(\mathbb{R}^d)$ supported in $\CB$ and $\CC$, respectively, such that
\begin{align*}
&\chi(\xi)+\sum_{j\geq0}\varphi(2^{-j}\xi)=1,~~\forall \xi\in\mathbb{R}^d,\\
&|j-j'|\geq2\Rightarrow \text{Supp}~\varphi(2^{-j}\xi)\cap \text{Supp}~\varphi(2^{-j'}\xi)=\emptyset.
\end{align*}
Denoting by $\MF$ and $\MF^{-1}$ the Fourier transform and its inverse, respectively, the frequency localization operators $\Delta_j$ and $S_j$ can be defined by
\begin{align*}
\Delta_jf\eqdefa&\MF^{-1}(\varphi(2^{-j}\xi)\MF f)=2^{jd}\int_{\mathbb{R}^d}h(2^j\yy)f(\xx-\yy)\ud\yy,~~\text{for}~j\geq0,\\
S_jf\eqdefa&\MF^{-1}(\chi(2^{-j}\xi)\MF f)=\sum_{-1\leq k\leq j-1}\Delta_k f=2^{jd}\int_{\mathbb{R}^d}\tilde{h}(2^j\yy)f(\xx-\yy)\ud\yy,\\
\Delta_{-1}f=&S_0f,\qquad \Delta_jf=0,~~\text{for}~j\leq-2,
\end{align*}
where $h=\MF^{-1}\varphi$ and $\tilde{h}=\MF^{-1}\chi$. By the choice of $\varphi$, it can be proved that
\begin{align*}
&\Delta_j\Delta_kf=0,~~\text{if}~|j-k|\geq2,\\
&\Delta_j(S_{k-1}f\Delta_kf)=0,~~\text{if}~|j-k|\geq5.
\end{align*}

Let $s\in\mathbb{R}$ and $1\leq p,q\leq\infty$. By the operator $\Delta_j$, we can define the norm of an element $f$ in the nonhomogeneous Besov space $B^s_{p,q}$ as
\begin{align*}
\|f\|_{B^s_{p,q}}\eqdefa\big\|\{2^{js}\|\Delta_jf\|_{L^p}\}_{j\geq-1}\big\|_{\ell^q},\quad \|f\|_{B^s_{p,\infty}}\eqdefa\sup_{j\geq-1}\{2^{js}\|\Delta_jf\|_{L^p}\}.
\end{align*}
In particular, a function $f\in H^s$ is characterized as follows:
\begin{align*}
\|f\|_{H^s}\sim \big\|\{2^{js}\|\Delta_jf\|_{L^2}\}_{j\geq-1}\big\|_{\ell^2}.
\end{align*}

For two smooth functions $u$ and $v$, the Bony's paraproduct decomposition in nonhomogeneous case is defined by
\begin{align*}
uv=T_uv+T_vu+R(u,v),
\end{align*}
where
\begin{align*}
T_uv=\sum_jS_{j-1}u\Delta_jv,\quad R(u,v)=\sum_{|j-j'|\leq 1}\Delta_ju\Delta_j'v.
\end{align*}

\subsection{Some useful lemmas}

Let us now introduce some useful lemmas which will be  frequently used later.
\begin{lemma}[Berstein's inequalities \cite{BCD}]\label{b-inequalities}
Assume that $1 \leq p \leq q \leq \infty$ and $f \in L^{p}\left(\mathbb{R}^{d}\right)$. Then it follows that
\begin{align*}
&\operatorname{Supp} \widehat{f} \subset\big\{|\xi| \leq C2^j\big\} \Rightarrow\left\|\partial^{\alpha} f\right\|_{L^q} \leq C 2^{j|\alpha|+dj\left(\frac{1}{p}-\frac{1}{q}\right)}\|f\|_{L^p}, \\
&\operatorname{Supp} \widehat{f} \subset\Big\{\frac{1}{C}2^j \leq|\xi| \leq C2^j\Big\} \Rightarrow\|f\|_{L^p} \leq C 2^{-j|\alpha|} \sup _{|\beta|=|\alpha|}\left\|\partial^{\beta} f\right\|_{L^p},
\end{align*}
where the constant $C$ is independent of $f$ and $j$.
\end{lemma}

\begin{lemma}[\cite{WWZ-zju}]\label{useful-inequality}
Let $s\in(0,1)$. For any $j \geq-1$, it follows that
\begin{align*}
\left\|\Delta_{j}(f g)\right\|_{L^2} \leq& C 2^{j s}\|f\|_{B_{2, \infty}^{-s}}\|g\|_{H^{1}}+C 2^{\frac{(s+1)j}{2}}\|g\|_{L^4}\|f\|_{B_{2, \infty}^{-s}}^{\frac{1}{2}} \sum_{\left|j^{\prime}-j\right| \leq 4}\left\|\Delta_{j^{\prime}} f\right\|_{L^2}^{\frac{1}{2}},\\
\|\Delta_j(fgh)\|_{L^2}\leq&C2^{js}\big(\|f\|_{L^{\infty}}+\|\nabla f\|_{L^2}\big)\|g\|_{B^{1-s}_{2,\infty}}\|h\|_{L^2},\\
\left\|\Delta_{j}(f \nabla g h)\right\|_{L^2} \leq& C 2^{j s}\|g\|_{B_{2, \infty}^{1-s}}\Big(\|f\|_{L^{\infty}}\|h\|_{H^{1}}+\|\nabla f\|_{L^4}\|h\|_{L^4}+\left\|\nabla^{2} f\right\|_{L^2}\|h\|_{L^2}\Big) \\
&+C 2^{\frac{j s}{2}}\|f\|_{L^{\infty}}\|h\|_{L^4}\|g\|_{B_{2, \infty}^{1-s}}^{\frac{1}{2}} \sum_{l=j-9}^{j+9} 2^{\frac{l}{2}}\left\|\Delta_{l} \nabla g\right\|_{L^2}^{\frac{1}{2}}.
\end{align*}
\end{lemma}

The proof of Lemma \ref{useful-inequality} mainly relies on the Bony's paraproduct decomposition in nonhomogeneous case. We can refer to \cite{WWZ-zju} for details.

\begin{lemma}[commutator estimate \cite{WWZ-zju}]\label{commutator}
Let $s \in(0,1)$. For any $j \geq-1$, it follows that
\begin{align*}
\left\|[\Delta_j,f]\nabla g\right\|_{L^2} \leq & C 2^{\frac{js}{2}}\|\nabla f\|_{L^4}\|g\|^{\frac{1}{2}}_{B^{-s}_{2,\infty}}\sum_{|j'-j|\leq4}2^{\frac{j'}{2}}\|\Delta_{j'}g\|^{\frac{1}{2}}_{L^2}\\
&+C2^{js}\|g\|_{B^{-s}_{2,\infty}}(\|f\|_{L^{\infty}}+\|\nabla^2f\|_{L^2}).
\end{align*}
\end{lemma}

The proof of uniqueness relies on the so-called Osgood-type inequality, a generalization of the Gronwall inequality. For convenience, we provide the so-called Osgood lemma (see Lemma 3.4 in \cite{BCD}).
\begin{definition}[Osgood condition \cite{BCD}]
Let $b>0$ and $\mu$ be modulus of continuity defined on $[0,b]$. We say that $\mu$ is an Osgood modulus of continuity if
\begin{align*}
    \int^b_0\frac{\ud r}{\mu(r)}=\infty.
\end{align*}
\end{definition}

\begin{lemma}[Osgood lemma \cite{BCD}]\label{Osgood-lemma}
Let $\Phi:[0,T]\rightarrow[0,\infty)$ be a measurable function and $F$ be a locally integrable function
from $[0,T]$ to $[0,\infty)$. Assume that there exists an Osgood modulus of continuity $\mu:[0,\infty)\rightarrow[0,\infty)$ such that
\begin{align*}
    \Phi(t)\leq\int^t_0F(\tau)\mu(\Phi(\tau))\ud\tau,~~\text{for almost all}~t\in[0,T],
\end{align*}
then $\Phi=0$ a.e. in $[0,T]$.
\end{lemma}

For the sake of analysis, we need to rewrite the orientational elasticity in (\ref{elastic-energy}). For any frame $\Fp=(\nn_1,\nn_2,\nn_3)\in SO(3)$, we have the following simple identity relations:
\begin{align*}
&|\nn_1\times(\nabla\times\nn_1)|^2=(\nn_2\cdot\nabla\times\nn_1)^2+(\nn_3\cdot\nabla\times\nn_1)^2,\\
&|\nn_2\times(\nabla\times\nn_2)|^2=(\nn_1\cdot\nabla\times\nn_2)^2+(\nn_3\cdot\nabla\times\nn_2)^2,\\
&|\nn_3\times(\nabla\times\nn_3)|^2=(\nn_1\cdot\nabla\times\nn_3)^2+(\nn_2\cdot\nabla\times\nn_3)^2,\\
&|\nabla\nn_i|^2=(\nabla\cdot\nn_i)^2+(\nn_i\cdot\nabla\times\nn_i)^2+|\nn_i\times(\nabla\times\nn_i)|^2\\
&\qquad\qquad+\nabla\cdot[(\nn_i\cdot\nabla)\nn_i-(\nabla\cdot\nn_i)\nn_i],\quad i=1,2,3.
\end{align*}
With
the aid of the above identity relations, the density $f_{Bi}(\Fp,\nabla\Fp)$ can be rewritten as
\begin{align}\label{new-elasitic-density}
    f_{Bi}(\Fp,\nabla\Fp)=\frac{1}{2}\sum^3_{i=1}\gamma_i|\nabla \nn_{i}|^2+W(\Fp,\nabla\Fp).
\end{align}
Here, the coefficients $\gamma_i(i=1,2,3)$ are taken as, respectively,
\begin{align}\label{gamma-123}
\left\{
    \begin{aligned}
&\gamma_1=\min\{K_1,K_4,K_7,K_{10}\}>0,~ \gamma_2=\min\{K_2,K_5,K_8,K_{11}\}>0,\\
&\gamma_3=\min\{K_3,K_6,K_9,K_{12}\}>0,
\end{aligned}
    \right.
\end{align}
and $W(\Fp,\nabla\Fp)$ is expressed by
\begin{align*}
W(\Fp,\nabla\Fp)=&\frac{1}{2}\Big(\sum^3_{i=1}k_i(\nabla\cdot\nn_i)^2+\sum^3_{i,j=1}k_{ij}(\nn_i\cdot\nabla\times\nn_j)^2\Big),
\end{align*}
where the coefficients $k_i\geq 0,k_{ij}\geq0(i,j=1,2,3)$ are given by
\begin{align}\label{ki-kij}
\left\{
    \begin{aligned}
&k_1=K_1-\gamma_1,\quad k_2=K_2-\gamma_2,\quad
k_3=K_3-\gamma_3,\\
&k_{11}=K_4-\gamma_1,\quad
k_{22}=K_5-\gamma_2,\quad
k_{33}=K_6-\gamma_3,\\
&k_{31}=K_7-\gamma_1,\quad
k_{12}=K_8-\gamma_2,\quad
k_{23}=K_9-\gamma_3,\\
&k_{21}=K_{10}-\gamma_1,\quad
k_{32}=K_{11}-\gamma_2,\quad
k_{13}=K_{12}-\gamma_3.
\end{aligned}
    \right.
\end{align}

For simplicity, we define
\begin{align*}
&(\hh_1,\hh_2,\hh_3)\eqdefa-\Big(\frac{\delta\CF_{Bi}}{\delta\nn_1},\frac{\delta\CF_{Bi}}{\delta\nn_2},\frac{\delta\CF_{Bi}}{\delta\nn_3}\Big)=-\frac{\delta\CF_{Bi}}{\delta\Fp}=\nabla\cdot\frac{\partial f_{Bi}}{\partial(\nabla\Fp)}-\frac{\partial f_{Bi}}{\partial\Fp}.
\end{align*}

Similar to Lemma 2.3 in \cite{WW}, we also introduce the algebraic structures of the  variational derivative with regards to the frame field $\Fp=(\nn_1,\nn_2,\nn_3)\in SO(3)$.

\begin{lemma}[\cite{LWX}]\label{h-decomposition}
For the terms $\hh_i(i=1,2,3)$, we have the following representation:
\begin{align*}
\hh_i=&\gamma_i\Delta\nn_i+k_i\nabla{\rm div}\nn_i
-\sum^3_{j=1}k_{ji}\nabla\times(\nabla\times\nn_i\cdot\nn^2_j)
-\sum^3_{j=1}k_{ij}(\nn_i\cdot\nabla\times\nn_j)(\nabla\times\nn_j),
\end{align*}
where $\nn^2_j=\nn_j\otimes\nn_j$ and the coefficients are expressed by \eqref{gamma-123} and \eqref{ki-kij}.
\end{lemma}

To handle the body force $\mathfrak{F}$ in analysis more conveniently, it is necessary to rewrite $\mathfrak{F}$ in an equivalent form. More specifically, we have the following lemma.

\begin{lemma}\label{body-force-lemma}
For any frame $\Fp=(\nn_1,\nn_2,\nn_3)\in SO(3)$, it follows that
\begin{align*}
(\mathfrak{F})_i=\sum^3_{\alpha=1}\partial_i\nn_{\alpha}\cdot\frac{\delta\CF_{Bi}}{\delta\nn_{\alpha}}\eqdefa\partial_j\sigma^d_{ij}+\partial_i\widetilde{p},
\end{align*}
where $\widetilde{p}$ can be absorbed into the pressure term $p$ in \eqref{frame-equation-v} and the elastic energy $\CF_{Bi}$ is given by \eqref{elastic-energy}, and the stress $\sigma^d_{ij}=-\frac{\partial f_{Bi}}{\partial(\partial_j\Fp)}\cdot\partial_i\Fp$.
\end{lemma}

\begin{proof}
We will use the Kronecker $\delta$ symbol.
Recalling the definitions of $\ML_k\CF_{Bi}(k=1,2,3)$ and using (\ref{import-properties}), together with the relation $\Fi=\nn^2_1+\nn^2_2+\nn^2_3$, we can derive that
\begin{align*}
(\mathfrak{F})_i=&\,\partial_i\nn_1\cdot\nn_2\Big(\nn_2\cdot\frac{\delta\CF_{Bi}}{\delta\nn_1}-\nn_1\cdot\frac{\delta\CF_{Bi}}{\delta\nn_2}\Big)+\partial_i\nn_3\cdot\nn_1\Big(\nn_1\cdot\frac{\delta\CF_{Bi}}{\delta\nn_3}-\nn_3\cdot\frac{\delta\CF_{Bi}}{\delta\nn_1}\Big)\\
&+\partial_i\nn_2\cdot\nn_3\Big(\nn_3\cdot\frac{\delta\CF_{Bi}}{\delta\nn_2}-\nn_2\cdot\frac{\delta\CF_{Bi}}{\delta\nn_3}\Big)\\
=&\,\partial_in_{1k}(n_{2k}n_{2l}+n_{3k}n_{3l})\frac{\delta\CF_{Bi}}{\delta n_{1l}}+\partial_in_{2k}(n_{1k}n_{1l}+n_{3k}n_{3l})\frac{\delta\CF_{Bi}}{\delta n_{2l}}\\
&+\partial_in_{3k}(n_{1k}n_{1l}+n_{2k}n_{2l})\frac{\delta\CF_{Bi}}{\delta n_{3l}}\\
=&\,\partial_in_{1k}(\delta_{kl}-n_{1k}n_{1l})\frac{\delta\CF_{Bi}}{\delta n_{1l}}+\partial_in_{2k}(\delta_{kl}-n_{2k}n_{2l})\frac{\delta\CF_{Bi}}{\delta n_{2l}}\\
&+\partial_in_{3k}(\delta_{kl}-n_{3k}n_{3l})\frac{\delta\CF_{Bi}}{\delta n_{3l}}\\
=&\,\partial_in_{1k}\frac{\delta\CF_{Bi}}{\delta n_{1k}}+\partial_in_{2k}\frac{\delta\CF_{Bi}}{\delta n_{2k}}+\partial_in_{3k}\frac{\delta\CF_{Bi}}{\delta n_{3k}}\\
\eqdefa&\,\partial_j\sigma^d_{ij}+\partial_i\widetilde{p},
\end{align*}
where $\sigma_{ij}=-\sum^3_{\alpha=1}\frac{\partial f_{Bi}}{\partial(\partial_j\nn_{\alpha})}\cdot\partial_i\nn_{\alpha}$.
Moreover, by a direct calculation,  we obtain
\begin{align}\label{sigma-d}
\sigma_{ij}^d(\nabla\Fp,\Fp) &=-\sum_{\alpha=1}^3\gamma_{\alpha}\partial_jn_{\alpha k}\partial_in_{\alpha k}-\sum_{\alpha=1}^3 k_{\alpha}(\nabla\cdot\nn_{\alpha})\partial_in_{\alpha j}\nonumber\\
&\quad-\sum_{\alpha,\beta=1}^3 k_{\beta\alpha}\Big[(\partial_jn_{\alpha p}-\partial_pn_{\alpha j})\partial_in_{\alpha p}+n_{\beta j}n_{\beta l}(\partial_pn_{\alpha l}-\partial_ln_{\alpha p})\partial_in_{\alpha p}\nonumber\\
&\quad+n_{\beta p}n_{\beta l}(\partial_ln_{\alpha j}-\partial_jn_{\alpha l})\partial_in_{\alpha p}\Big].
\end{align}
\end{proof}

\section{Uniqueness of weak solutions}\label{uniqueness-section}

This section is devoted to the proof of the uniqueness of weak solutions in the class $V(0,T)\times H(0,T)$.
Let $(\Fp^{(1)},\vv^{(1)})$ and $(\Fp^{(2)},\vv^{(2)})$ be two weak solutions of the frame hydrodynamic system (\ref{new-frame-equation-n1})--(\ref{imcompressible-v}) with the same initial data $(\Fp_0,\vv_0)$, where the orthonormal frames $\Fp^{(i)}=\big(\nn^{(i)}_1,\nn^{(i)}_2,\nn^{(i)}_3\big) (i=1,2)$.

We denote
\begin{align*}
&\delta_{\nn_j}=\nn^{(1)}_j-\nn^{(2)}_j,\quad \delta_{\hh_j}=\hh^{(1)}_j-\hh^{(2)}_j,\quad \delta_{\aaa_j}=\aaa^{(1)}_j-\aaa^{(2)}_j,~~ (j=1,2,3),\\
&\delta_{\vv}=\vv^{(1)}-\vv^{(2)},\quad \delta_{\A}=\A^{(1)}-\A^{(2)},\quad \delta_{\BOm}=\BOm^{(1)}-\BOm^{(2)},\quad \delta_p=p^{(1)}-p^{(2)},\\
&\delta_{\ML_1}=\nn^{(1)}_2\cdot\delta_{\hh_3}-\nn^{(1)}_3\cdot\delta_{\hh_2},\quad\delta_{\ML_2}=\nn^{(1)}_3\cdot\delta_{\hh_1}-\nn^{(1)}_1\cdot\delta_{\hh_3},\\
&\delta_{\ML_3}=\nn^{(1)}_1\cdot\delta_{\hh_2}-\nn^{(1)}_2\cdot\delta_{\hh_1},\quad \delta_{\sss_k}=\sss^{(1)}_k-\sss^{(2)}_k,~~ (k=1,\cdots,5).
\end{align*}
By taking the difference between the equations for $\big(\Fp^{(1)},\vv^{(1)}\big)$ and $\big(\Fp^{(2)},\vv^{(2)}\big)$, we find that
\begin{align}
\frac{\partial\delta_{\nn_1}}{\partial t}=&\Big(\frac{1}{2}\delta_{\BOm}\cdot\aaa^{(1)}_1+\frac{\eta_3}{\chi_3}\delta_{\A}\cdot\sss^{(1)}_3-\frac{1}{\chi_3}\delta_{\ML_3}\Big)\nn^{(1)}_2\nonumber\\
&-\Big(\frac{1}{2}\delta_{\BOm}\cdot\aaa^{(1)}_2+\frac{\eta_2}{\chi_2}\delta_{\A}\cdot\sss^{(1)}_4-\frac{1}{\chi_2}\delta_{\ML_2}\Big)\nn^{(1)}_3+\delta_{\FF_1},\label{n1-uniquence-equ}\\
\frac{\partial\delta_{\nn_2}}{\partial t}=&-\Big(\frac{1}{2}\delta_{\BOm}\cdot\aaa^{(1)}_1+\frac{\eta_3}{\chi_3}\delta_{\A}\cdot\sss^{(1)}_3-\frac{1}{\chi_3}\delta_{\ML_3}\Big)\nn^{(1)}_1\nonumber\\
&+\Big(\frac{1}{2}\delta_{\BOm}\cdot\aaa^{(1)}_3+\frac{\eta_1}{\chi_1}\delta_{\A}\cdot\sss^{(1)}_5-\frac{1}{\chi_1}\delta_{\ML_1}\Big)\nn^{(1)}_3+\delta_{\FF_2},\label{n2-uniquence-equ}\\
\frac{\partial\delta_{\nn_3}}{\partial t}=&\Big(\frac{1}{2}\delta_{\BOm}\cdot\aaa^{(1)}_2+\frac{\eta_2}{\chi_2}\delta_{\A}\cdot\sss^{(1)}_4-\frac{1}{\chi_2}\delta_{\ML_2}\Big)\nn^{(1)}_1\nonumber\\
&-\Big(\frac{1}{2}\delta_{\BOm}\cdot\aaa^{(1)}_3+\frac{\eta_1}{\chi_1}\delta_{\A}\cdot\sss^{(1)}_5-\frac{1}{\chi_1}\delta_{\ML_1}\Big)\nn^{(1)}_2+\delta_{\FF_3},\label{n3-uniquence-equ}\\
\frac{\partial\delta_{\vv}}{\partial t}=&-\nabla\delta_{p}+\eta\Delta\delta_{\vv}+\nabla\cdot\big(\sigma(\Fp^{(1)},\delta_{\vv})+\sigma^d(\nabla\delta_{\Fp})\big)+\nabla\cdot\delta_{\FF_4},\label{v-uniquence-equ}\\
\nabla\cdot\delta_{\vv}=&0,\label{incomp-uniquence-equ}
\end{align}
where the stresses $\sigma(\Fp^{(1)},\delta_{\vv})$ and $\sigma^d(\nabla\delta_{\Fp})$ are expressed by
\begin{align*}
\sigma(\Fp^{(1)},\delta_{\vv})=&\beta_1(\delta_{\A}\cdot\sss^{(1)}_1)\sss^{(1)}_1+\beta_0(\delta_{\A}\cdot\sss^{(1)}_2)\sss^{(1)}_1+\beta_0(\delta_{\A}\cdot\sss^{(1)}_1)\sss^{(1)}_2\\
&+\beta_2(\delta_{\A}\cdot\sss^{(1)}_2)\sss^{(1)}_2+\Big(\beta_3-\frac{\eta^2_3}{\chi_3}\Big)(\delta_{\A}\cdot\sss^{(1)}_3)\sss^{(1)}_3+\Big(\beta_4-\frac{\eta^2_2}{\chi_2}\Big)(\delta_{\A}\cdot\sss^{(1)}_4)\sss^{(1)}_4\\
&+\Big(\beta_5-\frac{\eta^2_1}{\chi_1}\Big)(\delta_{\A}\cdot\sss^{(1)}_5)\sss^{(1)}_5+\frac{\eta_3}{\chi_3}\sss^{(1)}_3\delta_{\ML_3}+\frac{\eta_2}{\chi_2}\sss^{(1)}_4\delta_{\ML_2}+\frac{\eta_1}{\chi_1}\sss^{(1)}_5\delta_{\ML_1}\\
&-\frac{1}{2}\big(\aaa^{(1)}_1\delta_{\ML_3}+\aaa^{(1)}_2\delta_{\ML_2}+\aaa^{(1)}_3\delta_{\ML_1}\big),\\
\sigma^d(\nabla\delta_{\Fp})=&\sigma^d(\nabla\Fp^{(1)},\Fp^{(1)})-\sigma^d(\nabla\Fp^{(2)},\Fp^{(2)}).
\end{align*}
Moreover, $\delta_{\FF_i}(i=1,\cdots,5)$ can be given by
\begin{align*}
\delta_{\FF_1}=&\Big(\frac{1}{2}\BOm^{(2)}\cdot\delta_{\aaa_1}+\frac{\eta_3}{\chi_3}\A^{(2)}\cdot\delta_{\sss_3}-\frac{1}{\chi_3}(\delta_{\nn_1}\cdot\hh^{(2)}_2-\delta_{\nn_2}\cdot\hh^{(2)}_1)\Big)\nn^{(1)}_2\nonumber\\
&+\Big(\frac{1}{2}\BOm^{(2)}\cdot\aaa^{(2)}_1+\frac{\eta_3}{\chi_3}\A^{(2)}\cdot\sss^{(2)}_3-\frac{1}{\chi_3}\ML^{(2)}_3\CF_{Bi}\Big)\delta_{\nn_2}\nonumber\\
&-\Big(\frac{1}{2}\BOm^{(2)}\cdot\delta_{\aaa_2}+\frac{\eta_2}{\chi_2}\A^{(2)}\cdot\delta_{\sss_4}-\frac{1}{\chi_2}(\delta_{\nn_3}\cdot\hh^{(2)}_1-\delta_{\nn_1}\cdot\hh^{(2)}_3)\Big)\nn^{(1)}_3\nonumber\\
&-\Big(\frac{1}{2}\BOm^{(2)}\cdot\aaa^{(2)}_2+\frac{\eta_2}{\chi_2}\A^{(2)}\cdot\sss^{(2)}_4-\frac{1}{\chi_2}\ML^{(2)}_2\CF_{Bi}\Big)\delta_{\nn_3}\\
&-\vv^{(1)}\cdot\nabla\delta_{\nn_1}-\delta_{\vv}\cdot\nabla\nn^{(2)}_1,\\
\delta_{\FF_2}=&-\Big(\frac{1}{2}\BOm^{(2)}\cdot\delta_{\aaa_1}+\frac{\eta_3}{\chi_3}\A^{(2)}\cdot\delta_{\sss_3}-\frac{1}{\chi_3}(\delta_{\nn_1}\cdot\hh^{(2)}_2-\delta_{\nn_2}\cdot\hh^{(2)}_1)\Big)\nn^{(1)}_1\nonumber\\
&-\Big(\frac{1}{2}\BOm^{(2)}\cdot\aaa^{(2)}_1+\frac{\eta_3}{\chi_3}\A^{(2)}\cdot\sss^{(2)}_3-\frac{1}{\chi_3}\ML^{(2)}_3\CF_{Bi}\Big)\delta_{\nn_1}\nonumber\\
&+\Big(\frac{1}{2}\BOm^{(2)}\cdot\delta_{\aaa_3}+\frac{\eta_1}{\chi_1}\A^{(2)}\cdot\delta_{\sss_5}-\frac{1}{\chi_1}(\delta_{\nn_2}\cdot\hh^{(2)}_3-\delta_{\nn_3}\cdot\hh^{(2)}_2)\Big)\nn^{(1)}_3\nonumber\\
&+\Big(\frac{1}{2}\BOm^{(2)}\cdot\aaa^{(2)}_3+\frac{\eta_1}{\chi_1}\A^{(2)}\cdot\sss^{(2)}_5-\frac{1}{\chi_1}\ML^{(2)}_1\CF_{Bi}\Big)\delta_{\nn_3}\\
&-\vv^{(1)}\cdot\nabla\delta_{\nn_2}-\delta_{\vv}\cdot\nabla\nn^{(2)}_2,\\
\delta_{\FF_3}=&\Big(\frac{1}{2}\BOm^{(2)}\cdot\delta_{\aaa_2}+\frac{\eta_2}{\chi_2}\A^{(2)}\cdot\delta_{\sss_4}-\frac{1}{\chi_2}(\delta_{\nn_3}\cdot\hh^{(2)}_1-\delta_{\nn_1}\cdot\hh^{(2)}_3)\Big)\nn^{(1)}_1\nonumber\\
&+\Big(\frac{1}{2}\BOm^{(2)}\cdot\aaa^{(2)}_2+\frac{\eta_2}{\chi_2}\A^{(2)}\cdot\sss^{(2)}_4-\frac{1}{\chi_2}\ML^{(2)}_2\CF_{Bi}\Big)\delta_{\nn_1}\nonumber\\
&-\Big(\frac{1}{2}\BOm^{(2)}\cdot\delta_{\aaa_3}+\frac{\eta_1}{\chi_1}\A^{(2)}\cdot\delta_{\sss_5}-\frac{1}{\chi_1}(\delta_{\nn_2}\cdot\hh^{(2)}_3-\delta_{\nn_3}\cdot\hh^{(2)}_2)\Big)\nn^{(1)}_2\nonumber\\
&-\Big(\frac{1}{2}\BOm^{(2)}\cdot\aaa^{(2)}_3+\frac{\eta_1}{\chi_1}\A^{(2)}\cdot\sss^{(2)}_5-\frac{1}{\chi_1}\ML^{(2)}_1\CF_{Bi}\Big)\delta_{\nn_2}\\
&-\vv^{(1)}\cdot\nabla\delta_{\nn_1}-\delta_{\vv}\cdot\nabla\nn^{(2)}_1,\\
\delta_{\FF_4}=&-\vv^{(1)}\otimes\delta_{\vv}-\delta_{\vv}\otimes\vv^{(2)}
+\beta_1\big((\A^{(2)}\cdot\delta_{\sss_1})\sss^{(1)}_1+(\A^{(2)}\cdot\sss^{(2)}_1)\delta_{\sss_1}\big)\\
&+\beta_0\big((\A^{(2)}\cdot\delta_{\sss_2})\sss^{(1)}_1+(\A^{(2)}\cdot\sss^{(2)}_2)\delta_{\sss_1}\big)
+\beta_0\big((\A^{(2)}\cdot\delta_{\sss_1})\sss^{(1)}_2+(\A^{(2)}\cdot\sss^{(2)}_1)\delta_{\sss_2}\big)\\
&+\beta_2\big((\A^{(2)}\cdot\delta_{\sss_2})\sss^{(1)}_2+(\A^{(2)}\cdot\sss^{(2)}_2)\delta_{\sss_2}\big)
+\Big(\beta_3-\frac{\eta^2_3}{\chi_3}\Big)\big((\A^{(2)}\cdot\delta_{\sss_3})\sss^{(1)}_3+(\A^{(2)}\cdot\sss^{(2)}_3)\delta_{\sss_3}\big)\\
&+\Big(\beta_4-\frac{\eta^2_2}{\chi_2}\Big)\big((\A^{(2)}\cdot\delta_{\sss_4})\sss^{(1)}_4+(\A^{(2)}\cdot\sss^{(2)}_4)\delta_{\sss_4}\big)\\
&+\Big(\beta_5-\frac{\eta^2_1}{\chi_1}\Big)\big((\A^{(2)}\cdot\delta_{\sss_5})\sss^{(1)}_5+(\A^{(2)}\cdot\sss^{(2)}_5)\delta_{\sss_5}\big)\\
&+\frac{\eta_3}{\chi_3}\big(\sss^{(1)}_3(\delta_{\nn_1}\cdot\hh^{(2)}_2-\delta_{\nn_2}\cdot\hh^{(2)}_1)+\delta_{\sss_3}(\ML^{(2)}_3\CF_{Bi})\big)\\
&+\frac{\eta_2}{\chi_2}\big(\sss^{(1)}_4(\delta_{\nn_3}\cdot\hh^{(2)}_1-\delta_{\nn_1}\cdot\hh^{(2)}_3)+\delta_{\sss_4}(\ML^{(2)}_2\CF_{Bi})\big)\\
&+\frac{\eta_1}{\chi_1}\big(\sss^{(1)}_5(\delta_{\nn_2}\cdot\hh^{(2)}_3-\delta_{\nn_3}\cdot\hh^{(2)}_2)+\delta_{\sss_5}(\ML^{(2)}_1\CF_{Bi})\big)\\
&-\frac{1}{2}\big(\aaa^{(1)}_1(\delta_{\nn_1}\cdot\hh^{(2)}_2-\delta_{\nn_2}\cdot\hh^{(2)}_1)+\delta_{\aaa_1}(\ML^{(2)}_3\CF_{Bi})\big)\\
&-\frac{1}{2}\big(\aaa^{(1)}_2(\delta_{\nn_3}\cdot\hh^{(2)}_1-\delta_{\nn_1}\cdot\hh^{(2)}_3)+\delta_{\aaa_2}(\ML^{(2)}_2\CF_{Bi})\big)\\
&-\frac{1}{2}\big(\aaa^{(1)}_3(\delta_{\nn_2}\cdot\hh^{(2)}_3-\delta_{\nn_3}\cdot\hh^{(2)}_2)+\delta_{\aaa_3}(\ML^{(2)}_1\CF_{Bi})\big),
\end{align*}
where $\ML^{(2)}_k\CF_{Bi}(k=1,2,3)$ are denoted as, respectively,
\begin{align*}
&\ML^{(2)}_1\CF_{Bi}=\nn^{(2)}_2\cdot\hh^{(2)}_3-\nn^{(2)}_3\cdot\hh^{(2)}_2,~~\ML^{(2)}_2\CF_{Bi}=\nn^{(2)}_3\cdot\hh^{(2)}_1-\nn^{(2)}_1\cdot\hh^{(2)}_3,\\
&\ML^{(2)}_3\CF_{Bi}=\nn^{(2)}_1\cdot\hh^{(2)}_2-\nn^{(2)}_2\cdot\hh^{(2)}_1.
\end{align*}

Before presenting the proof the uniqueness, we need to discuss the suitable choice of the energy metric for the resulting system of the difference.
A usual idea to prove the uniqueness is to derive the energy estimates for the system (\ref{n1-uniquence-equ})--(\ref{incomp-uniquence-equ}) based on the following natural energy metric:
\begin{align*}
\|\delta_{\vv}(t)\|^2_{L^2}+\sum^3_{k=1}\|\nabla\delta_{\nn_k}(t)\|^2_{L^2}.
\end{align*}
However, in the sense of weak solutions, the above approach will no longer be available, since we would encounter nonlinear terms like $\int_{\mathbb{R}^2}\delta_{\FF_k}\cdot\Delta\delta_{\nn_k}\ud\xx (k=1,2,3)$, which cannot result in a closed energy estimate by using the embedding inequalities. A similar problem has been discussed in \cite{Xu-Zhangzf,LTX,WWZ-zju} for the Ericksen--Lesile model.

As a consequence, a key ingredient to control the difference $(\delta_{\Fp},\delta_{\vv})$ in the class $V(0,T)\times H(0,T)$ is to introduce a suitable weaker metric:
\begin{align}\label{Phi-t-metric}
\Phi(t)\eqdefa \CV(t)+\CU(t),
\end{align}
where for some $s\in (0,1/2)$, $\CV(t)$ and $\CU(t)$ are given by, respectively,
\begin{align*}
&\CV(t)=\sup\limits_{j \geq -1}2^{-2sj}\|\Delta_j\delta_{\vv}(t)\|_{L^2}^2,\quad \CU(t)= \sup\limits_{j \geq -1}2^{2(1-s)j}\|\Delta_j\delta_{\Fp}(t)\|_{L^2}^2,\\
&\|\Delta_j\delta_{\Fp}\|^2_{L^2}=\sum^3_{k=1}\|\Delta_j\delta_{\nn_k}\|^2_{L^2},~\text{for $j\ge -1$},\quad \|\Delta_{-1}\delta_{\Fp}\|_{L^2}^2=\sum\limits^3_{k=1}\|\Delta_{-1}\delta_{\nn_k}\|_{L^2}^2.
\end{align*}
Here $\Delta_j$ is the Littlewood--Paley operator defined in Subsection \ref{Little-Paley-subsect}.
Again, we introduce the following functional $\CW(t)$:
\begin{align*}
\CW(t)=\sup\limits_{j \geq -1}2^{-2sj}\int_{\mathbb{R}^2}W^j(\xx,t)\ud\xx+\|\Delta_{-1}\delta_{\Fp}(t)\|_{L^2}^2,
\end{align*}
where $W^j(\xx,t)$ is defined as
\begin{align}\label{Wj-xt}
W^j(\xx,t)\eqdefa\frac{1}{2}\Big(\sum_{i=1}^{3}\gamma_i\|\nabla\Delta_j\delta_{\nn_i}\|_{L^2}^2+\sum_{i=1}^{3}k_i\|\nabla\cdot\Delta_j\delta_{\nn_i}\|_{L^2}^2+\sum_{i,j=1}^{3}k_{ij}\|\nn_i^{(1)}\cdot\nabla\times\Delta_j\delta_{\nn_j}\|_{L^2}^2\Big),
\end{align}
where the coefficients $\gamma_i>0, k_i,k_{ij}\geq0 (i,j=1,2,3)$ are expressed by \eqref{gamma-123} and \eqref{ki-kij}.

Moreover, to close the energy estimate for the difference $(\delta_{\Fp},\delta_{\vv})$, we also need the following simple relation:
\begin{align}\label{Wj-relation}
\int_{\mathbb{R}^2}W^j(\xx,t)\ud\xx+\|\Delta_{-1}\delta_{\Fp}\|^2_{L^2}\geq c2^{2j}\|\Delta_j\delta_{\Fp}\|^2_{L^2},
\end{align}
which implies that $\CW(t)\geq c~\CU(t)$ for some constant $c>0$.

In order to present the proof of Theorem \ref{uniqueness-theorem} conveniently, we introduce a locally integrable function on $[0,T]$ as follows:
\begin{align*}
F(t)\eqdefa &1+\|(\nabla\vv^{(1)},\nabla\vv^{(2)})\|_{L^2}^2+\|(\vv^{(1)},\vv^{(2)})\|_{L^4}^4+\|(\nabla\Fp^{(1)},\nabla\Fp^{(2)})\|_{L^4}^4\\
&+\|\partial_t\Fp^{(1)}\|_{L^2}^2+\|(\nabla\Fp^{(1)},\nabla\Fp^{(2)})\|_{H^1}^2+\|(\nabla^2\Fp^{(1)},\nabla^2\Fp^{(2)})\|^2_{L^2},
\end{align*}
where $\Fp^{(\alpha)}=\big(\nn^{(\alpha)}_1,\nn^{(\alpha)}_2,\nn^{(\alpha)}_3\big)(\alpha=1,2)$.

We now turn to the poof of the uniqueness of weak solutions in $\mathbb{R}^2$. The proof will be divided into three subsequent subsections.

\subsection{Estimate for the velocity $\delta_{\vv}$}

We first notice that from Lemma \ref{b-inequalities}, there exists a constant $c>0$ such that
\begin{align}\label{nabla-vv-jvv}
ca_j2^{2j}\|\Delta_j\delta_{\vv}\|_{L^2}^2\leq\|\Delta_j\nabla\delta_{\vv}\|_{L^2}^2,
\end{align}
where $a_j=1$ for $j\geq0$ and $a_j=0$ for $j=-1$.

Applying the operator $\Delta_j$ to both sides of the equation (\ref{v-uniquence-equ}) and taking the inner product with $\Delta_j\delta_{\vv}$, together with (\ref{nabla-vv-jvv}),
we can deduce that
\begin{align}\label{difference-vv-L2}
&\frac{1}{2}\frac{\ud}{\ud t}\|\Delta_j\delta_{\vv}\|_{L^2}^2+ca_j2^{2j}\|\Delta_j\delta_{\vv}\|_{L^2}^2\nonumber\\
&\quad\leq-\bigg(\beta_1\|\sss^{(1)}_1\cdot\Delta_j\delta_{\A}\|_{L^2}^2+2\beta_0\int_{\mathbb{R}^2}(\Delta_j\delta_{\A}\cdot\sss^{(1)}_1)(\Delta_j\delta_{\A}\cdot\sss^{(1)}_2)\ud\xx+\beta_2\|\Delta_j\delta_{\A}\cdot\sss^{(1)}_2\|_{L^2}^2\bigg)\nonumber\\
&\qquad-\Big(\beta_3-\frac{\eta_{3}^2}{\chi_3}\Big)\|\Delta_j\delta_{\A}\cdot\sss^{(1)}_3\|_{L^2}^2-\Big(\beta_4-\frac{\eta_{2}^2}{\chi_2}\Big)\|\Delta_j\delta_{\A}\cdot\sss^{(1)}_4\|_{L^2}^2-\Big(\beta_5-\frac{\eta_{1}^2}{\chi_1}\Big)\|\Delta_j\delta_{\A}\cdot\sss^{(1)}_5\|_{L^2}^2\nonumber\\
&\qquad-\Big\langle\frac{\eta_3}{\chi_3}\sss^{(1)}_3\mathcal{H}^{\Delta_j}_3+\frac{\eta_2}{\chi_2}\sss^{(1)}_4\mathcal{H}^{\Delta_j}_2+\frac{\eta_1}{\chi_1}\sss^{(1)}_5\mathcal{H}^{\Delta_j}_1,\Delta_j\delta_{\A}\Big\rangle\nonumber\\
&\qquad+\Big\langle\frac{1}{2}\big(\aaa^{(1)}_1\mathcal{H}^{\Delta_j}_3+\aaa^{(1)}_2\mathcal{H}^{\Delta_j}_2+\aaa^{(1)}_3\mathcal{H}^{\Delta_j}_1\big),\Delta_j\delta_{\BOm}\Big\rangle\nonumber\\
&\qquad+\underbrace{\big\langle[\Delta_j,\mathcal{P}]\nabla\delta_{\vv},\Delta_j\nabla\delta_{\vv}\big\rangle}_{I_1}+\underbrace{\sum_{i=1}^{3}\big\langle[\Delta_j,\mathcal{P}]\delta_{\hh_i},\Delta_j\nabla\delta_{\vv}\big\rangle}_{I_2}\nonumber\\
&\qquad\underbrace{-\big\langle\Delta_j\sigma^d(\nabla\delta_{\Fp}),\Delta_j\nabla\delta_{\vv}\big\rangle}_{I_3}\underbrace{-\big\langle\Delta_j\delta_{\FF_4},\Delta_j\nabla\delta_{\vv}\big\rangle}_{I_4},
\end{align}
where the symbol $\CP$ stands for a polynomial function of  $(\Fp^{(1)},\Fp^{(2)})$ with degree no more than four, and we have used the relation $\Delta_j\nabla\delta_{\vv}=\Delta_j\delta_{\A}+\Delta_j\delta_{\BOm}$. Again, in (\ref{difference-vv-L2}) $\CH^{\Delta_j}_k(k=1,2,3)$ are defined as, respectively,
\begin{align}\label{H-Delta-j}
\left\{
\begin{aligned}
&\mathcal{H}^{\Delta_j}_1\eqdefa\nn^{(1)}_2\cdot\Delta_j\delta_{\hh_3}-\nn^{(1)}_3\cdot\Delta_j\delta_{\hh_2},\quad \mathcal{H}^{\Delta_j}_2\eqdefa\nn^{(1)}_3\cdot\Delta_j\delta_{\hh_1}-\nn^{(1)}_1\cdot\Delta_j\delta_{\hh_3},\\
&\mathcal{H}^{\Delta_j}_3\eqdefa\nn^{(1)}_1\cdot\Delta_j\delta_{\hh_2}-\nn^{(1)}_2\cdot\Delta_j\delta_{\hh_1}.
\end{aligned}
    \right.
\end{align}

For the term $I_1$, applying Lemma \ref{commutator} yields
\begin{align*}
I_1\leq& C\Big(2^{\frac{js}{2}}\|\nabla\CP\|_{L^4}\CV(t)^{\frac{1}{4}}\sum_{|j'-j|\leq4}2^{\frac{j'}{2}}\|\Delta_{j'}\delta_{\vv}\|^{\frac{1}{2}}_{L^2}\Big)\|\Delta_j\nabla\delta_{\vv}\|_{L^2}\nonumber\\
&+C2^{js}\CV(t)^{\frac{1}{2}}(1+\|\nabla^2\CP\|_{L^2})\|\Delta_j\nabla\delta_{\vv}\|_{L^2}\nonumber\\
\leq& C2^{2js}F(t)\CV(t)+\ve\sum_{|j'-j|\leq 4}2^{2j'}\|\Delta_{j'}\delta_{\vv}\|_{L^2}^2+\ve2^{2j}\|\Delta_j\delta_{\vv}\|_{L^2}^2,
\end{align*}
where we have employed the following simple facts:
\begin{align*}
&\|\nabla\CP\|_{L^4}\leq C\|(\nabla\Fp^{(1)},\nabla\Fp^{(2)})\|_{L^4},\quad \|(\vv^{(1)},\vv^{(2)})\|^2_{L^2}+\|(\nabla\Fp^{(1)},\nabla\Fp^{(2)})\|^2_{L^2}\leq C,\\
&\|\nabla^2\CP\|_{L^2}\leq C\|(\nabla\Fp^{(1)},\nabla\Fp^{(2)})\|^2_{L^4}+C\|(\nabla^2\Fp^{(1)},\nabla^2\Fp^{(2)})\|_{L^2}.
\end{align*}
In order to estimate the term $I_2$, it is necessary to express the difference $\delta_{\hh_i}(i=1,2,3)$ by introducing the short notation. Specifically, using Lemma \ref{h-decomposition}, we have
\begin{align}\label{delta-hh-i}
\delta_{\hh_i}=&\gamma_i\Delta\delta_{\nn_i}+k_i\nabla{\rm div}\delta_{\nn_i}-\sum^3_{j=1}k_{ji}\nabla\times\big(\nabla\times\delta_{\nn_i}\cdot\nn^{(1)}_j\otimes\nn^{(1)}_j\big)\nonumber\\
&-\sum^3_{j=1}k_{ji}\nabla\times\big(\nabla\times\nn^{(2)}_i\cdot\delta_{\nn_j}\otimes\nn^{(1)}_j\big)-\sum^3_{j=1}k_{ji}\nabla\times\big(\nabla\times\nn^{(2)}_i\cdot\nn^{(2)}_j\otimes\delta_{\nn_j}\big)\nonumber\\
&-\sum^3_{j=1}k_{ij}\big(\delta_{\nn_i}\cdot\nabla\times\nn^{(1)}_j\big)(\nabla\times\nn^{(1)}_j)-\sum^3_{j=1}k_{ij}\big(\nn^{(2)}_i\cdot\nabla\times\delta_{\nn_j}\big)(\nabla\times\nn^{(1)}_j)\nonumber\\
&-\sum^3_{j=1}k_{ij}\big(\nn^{(2)}_i\cdot\nabla\times\nn^{(2)}_j\big)(\nabla\times\delta_{\nn_j})\nonumber\\
\eqdefa&\CP\nabla^2\delta_{\nn_i}+\sum^3_{j=1}\CP\nabla^2\nn^{(2)}_i\delta_{\nn_j}+\sum^3_{j=1}\CP\nabla\nn^{(1)}_j\nabla\delta_{\nn_i}+\sum^3_{j=1}\CP\nabla\nn^{(2)}_i\nabla\delta_{\nn_j}\nonumber\\
&+\sum^3_{j=1}\sum^2_{\alpha=1}\CP\nabla\nn^{(\alpha)}_j\nabla\delta_{\nn_j}+\sum^3_{j=1}\sum^2_{\alpha=1}\CP\nabla\nn^{(2)}_i\nabla\nn^{(\alpha)}_j\delta_{\nn_j}+\sum^3_{j=1}\CP\nabla\nn^{(1)}_j\nabla\nn^{(1)}_j\delta_{\nn_i}.
\end{align}
Then, by using Lemma \ref{useful-inequality}, Lemma \ref{commutator} and the Sobolev's  embedding $H^1(\mathbb{R}^2)\hookrightarrow L^4(\mathbb{R}^2)$, we deduce from (\ref{delta-hh-i}) that
\begin{align*}
I_2\leq&C\sum_{i=1}^3\bigg[\|[\Delta_j,\CP^2]\nabla^2\delta_{\nn_i}\|_{L^2}+\sum^3_{k=1}\bigg(\|\Delta_j\big(\CP^2\nabla^2\nn_i^{(2)}\delta_{\nn_k}\big)\|_{L^2}+\|\Delta_j\big(\CP^2\nabla\nn^{(1)}_k\nabla\delta_{\nn_i}\big)\|_{L^2}\nonumber\\
&+\|\Delta_j\big(\CP^2\nabla\nn^{(2)}_i\nabla\delta_{\nn_k}\big)\|_{L^2}+\sum^2_{\alpha=1}\|\Delta_j\big(\CP^2\nabla\nn_i^{(\alpha)}\nabla\delta_{\nn_k}\big)\|_{L^2}\nonumber\\
&+\sum^2_{\alpha=1}\|\Delta_j\big(\CP^2\nabla\nn^{(2)}_i\nabla\nn^{(\alpha)}_k\delta_{\nn_k}\big)\|_{L^2}+\|\Delta_j\big(\CP^2\nabla\nn^{(1)}_k\nabla\nn^{(1)}_k\delta_{\nn_i}\big)\|_{L^2}\bigg)\bigg]\|\Delta_j\nabla\delta_{\vv}\|_{L^2}\nonumber\\
\leq&C2^{\frac{js}{2}}\Big(1+\|\nabla\CP^2\|_{L^4}+\|(\nabla\Fp^{(1)},\nabla\Fp^{(2)})\|_{L^4}\Big)\CU^{\frac{1}{4}}(t)\sum_{l=j-9}^{j+9}2^{\frac{l}{2}}\|\Delta_l\nabla\delta_{\Fp}\|_{L^2}^{\frac{1}{2}}\|\Delta_j\nabla\delta_{\vv}\|_{L^2}\nonumber\\
&+C2^{js}\bigg(1+\|\nabla^2\CP^2\|_{L^2}+\|\nabla^2\Fp^{(2)}\|_{L^2}+\|\nabla\CP^2\|_{L^4}\|(\nabla\Fp^{(1)},\nabla\Fp^{(2)})\|_{L^4}\nonumber\\
&+\|(\nabla\Fp^{(1)},\nabla\Fp^{(2)})\|_{H^1}+\|(\nabla\Fp^{(1)},\nabla\Fp^{(2)})\|_{L^4}^2\bigg)\CU^{\frac{1}{2}}(t)\|\Delta_j\nabla\delta_{\vv}\|_{L^2}\nonumber\\
\leq&C2^{\frac{js}{2}}F^{\frac{1}{4}}(t)\CU^{\frac{1}{4}}(t)\sum_{l=j-9}^{j+9}2^{\frac{l}{2}}\|\Delta_l\nabla\delta_{\Fp}\|_{L^2}^{\frac{1}{2}}\|\Delta_j\nabla\delta_{\vv}\|_{L^2}\nonumber\\
&+C2^{js}F^{\frac{1}{2}}(t)\CU^{\frac{1}{2}}(t)\|\Delta_j\nabla\delta_{\vv}\|_{L^2}\nonumber\\
\leq&\ve\sum_{l=j-9}^{j+9}2^{2l}\|\Delta_l\nabla\delta_{\Fp}\|_{L^2}^2+\ve2^{2j}\|\Delta_j\delta_{\vv}\|_{L^2}^2+C2^{2js}F(t)\CU(t),
\end{align*}
where $\ve$ denotes a small positive constant to be determined later, and we have noticed the fact that $\CP^2$ is a polynomial function of $(\Fp^{(1)},\Fp^{(2)})$ with degree no more than 8.

According to the expression of $\sigma^d$ in (\ref{sigma-d}), the stress $\sigma^d(\nabla\delta_{\Fp})$ can be expressed as
\begin{align*}
\sigma^d(\nabla\delta_{\Fp})=&\sigma^d(\nabla\Fp^{(1)},\Fp^{(1)})-\sigma^d(\nabla\Fp^{(2)},\Fp^{(2)})\\
=&\sum^3_{i,j=1}\sum^2_{\alpha=1}\CP\nabla\nn^{(\alpha)}_i\nabla\delta_{\nn_j}+\sum^3_{i,j,k=1}\CP\nabla\nn^{(1)}_i\nabla\nn^{(1)}_j\delta_{\nn_k}.
\end{align*}
Consequently, similar to the treatment of the term $I_2$, utilizing Lemma \ref{useful-inequality}, the term $I_3$ can be estimated as
\begin{align*}
I_3\leq& \Big(\sum^3_{i,j=1}\sum^2_{\alpha=1}\|\Delta_j\big(\CP\nabla\nn^{(\alpha)}_i\nabla\delta_{\nn_j}\big)\|_{L^2}+\sum^3_{i,j,k=1}\|\Delta_j\big(\CP\nabla\nn^{(1)}_i\nabla\nn^{(1)}_j\delta_{\nn_k}\big)\|_{L^2}\Big)\|\Delta_j\nabla\delta_{\vv}\|_{L^2}\\
\leq&C2^{2js}F(t)\CU(t)+\ve\sum_{l=j-9}^{j+9}2^{2l}\|\Delta_l\nabla\delta_{\Fp}\|_{L^2}^2+\ve2^{2j}\|\Delta_j\delta_{\vv}\|_{L^2}^2.
\end{align*}

To handle the term $I_4$, we first estimate the transport terms. Using Lemma \ref{useful-inequality}, we have
\begin{align}\label{transport-term-estimate}
&\big\langle\Delta_j\big(\delta_{\vv}\otimes\vv^{(2)}+\vv^{(1)}\otimes\delta_{\vv}\big),\Delta_j\nabla\delta_{\vv}\big\rangle\nonumber\\
&\quad\leq \Big(C2^{js}\|(\vv^{(1)},\vv^{(2)})\|_{H^1}\|\delta_{\vv}\|_{B^{-s}_{2,\infty}}\nonumber\\
&\qquad+C2^{\frac{(s+1)j}{2}}\|(\vv^{(1)},\vv^{(2)})\|_{L^4}\|\delta_{\vv}\|^{\frac{1}{2}}_{B^{-s}_{2,\infty}}\sum_{|j'-j|\leq4}\|\Delta_{j'}\delta_{\vv}\|^{\frac{1}{2}}_{L^2}\Big)\|\Delta_j\nabla\delta_{\vv}\|_{L^2}\nonumber\\
&\quad\leq C2^{2js}\|(\vv^{(1)},\vv^{(2)})\|_{H^1}^{2}\CV(t)+\ve2^{2j}\|\Delta_j\delta_{\vv}\|_{L^2}^2.
\end{align}
Then, combining (\ref{transport-term-estimate}) we obtain
\begin{align*}
I_4&\leq C2^{2js}\|(\vv^{(1)},\vv^{(2)})\|_{H^1}^{2}\CV(t)+\ve2^{2j}\|\Delta_j\delta_{\vv}\|_{L^2}^2\\
&\quad+\underbrace{\sum_{i=1}^{3}\big\langle\Delta_j(\CP\delta_{\nn_i}\nabla\vv^{(2)}),\Delta_j\nabla\delta_{\vv}\big\rangle}_{I_{41}}+\underbrace{\sum_{i,k=1}^{3}\big\langle\Delta_j(\CP\delta_{\nn_i}\hh^{(2)}_{k}),\Delta_j\nabla\delta_{\vv}\big\rangle}_{I_{42}}.
\end{align*}
Applying Lemma \ref{useful-inequality} and the expression of $\hh_i(i=1,2,3)$ in Lemma \ref{h-decomposition}, $I_{41}$ and $I_{42}$ can be estimated as, respectively,
\begin{align*}
I_{41}&\leq\sum^3_{i=1}\|\Delta_j(\CP\nabla\vv^{(2)}\delta_{\nn_i})\|_{L^2}\|\Delta_j\nabla\delta_{\vv}\|_{L^2}\\
&\leq C2^{js}(1+\|(\nabla\Fp^{(1)},\nabla\Fp^{(2)})\|_{L^2})\|\delta_{\Fp}\|_{B^{1-s}_{2,\infty}}\|\nabla\vv^{(2)}\|_{L^2}2^{j}\|\Delta_j\delta_{\vv}\|_{L^2}\\
&\leq C2^{2js}\|\nabla\vv^{(2)}\|_{L^2}^2\CU(t)+\ve2^{2j}\|\Delta_j\delta_{\vv}\|_{L^2}^2,\\
I_{42}&\leq\sum^3_{i,k=1}\|\Delta_j(\CP\delta_{\nn_i}\hh^{(2)}_k)\|_{L^2}\|\Delta_j\nabla\delta_{\vv}\|_{L^2}\\
&\leq C2^{js}(1+\|(\nabla\Fp^{(1)},\nabla\Fp^{(2)})\|_{L^2})\|\delta_{\Fp}\|_{B^{1-s}_{2,\infty}}\sum^3_{k=1}\|\hh^{(2)}_k\|_{L^2}2^{j}\|\Delta_j\delta_{\vv}\|_{L^2}\\
&\leq C2^{2js}\big(\|\nabla\Fp^{(2)}\|^4_{L^4}+\|\nabla^2\Fp^{(2)}\|^2_{L^2}\big)\CU(t)+\ve 2^{2j}\|\Delta_j\delta_{\vv}\|^2_{L^2}.
\end{align*}
Accordingly, combining the above estimates, we get
\begin{align*}
I_4\leq C2^{2js}F(t)(\CV(t)+\CU(t))+\ve 2^{2j}\|\Delta_j\delta_{\vv}\|^2_{L^2}.
\end{align*}

Thus, substituting the above estimates of $I_k(k=1,\cdots,4)$ into (\ref{difference-vv-L2}) and removing the dissipative terms, we arrive at
\begin{align}\label{differen-delta-vv-L2final}
&\frac{1}{2}\frac{\ud}{\ud t}\|\Delta_j\delta_{\vv}\|_{L^2}^2+ca_j2^{2j}\|\Delta_j\delta_{\vv}\|_{L^2}^2\nonumber\\
&\quad\leq -\Big\langle\frac{\eta_3}{\chi_3}\sss^{(1)}_3\mathcal{H}^{\Delta_j}_3+\frac{\eta_2}{\chi_2}\sss^{(1)}_4\mathcal{H}^{\Delta_j}_2+\frac{\eta_1}{\chi_1}\sss^{(1)}_5\mathcal{H}^{\Delta_j}_1,\Delta_j\delta_{\A}\Big\rangle\nonumber\\
&\qquad+\Big\langle\frac{1}{2}\big(\aaa^{(1)}_1\mathcal{H}^{\Delta_j}_3+\aaa^{(1)}_2\mathcal{H}^{\Delta_j}_2+\aaa^{(1)}_3\mathcal{H}^{\Delta_j}_1\big),\Delta_j\delta_{\BOm}\Big\rangle\nonumber\\
&\qquad +C2^{2js}F(t)\Phi(t)+\ve\sum_{l=j-9}^{j+9}2^{4l}\|\Delta_l\delta_{\Fp}\|_{L^2}^2+\ve\sum_{l=j-4}^{j+4}2^{2l}\|\Delta_l\delta_{\vv}\|_{L^2}^2.
\end{align}

\subsection{Estimate for the orthonormal frame $\delta_{\Fp}$}

This subsection is dedicated to the estimates of the difference for the orthonormal frame, i.e., $\delta_{\Fp}=\Fp^{(1)}-\Fp^{(2)}$. To begin with, we may verify a claim for $L^2$-estimate of $\Delta_{-1}\delta_{\Fp}$.

{\bf Claim 1}: {\it There exists a constant $C>0$, such that}
\begin{align}\label{claim-Delta-1-L2}
\frac{1}{2}\frac{\ud}{\ud t}\|\Delta_{-1}\delta_{\Fp}\|_{L^2}^2\leq C F(t)\big(\CV(t)+\CU(t)\big).
\end{align}

Indeed, applying the operator $\Delta_{-1}$ to both sides of the equation (\ref{n1-uniquence-equ}) and taking $L^2$-inner product with $\Delta_{-1}\delta_{\Fp}$, we can deduce that
\begin{align}\label{Delta--1Fp-L2}
&\frac{1}{2}\frac{\ud}{\ud t}\|\Delta_{-1}\delta_{\Fp}\|_{L^2}^2=\sum^3_{i=1}\frac{1}{2}\frac{\ud}{\ud t}\|\Delta_{-1}\delta_{\nn_i}\|_{L^2}^2\nonumber\\
&\quad=\sum^3_{i=1}\big\langle\Delta_{-1}(\CP\nabla\delta_{\vv}),\Delta_{-1}\delta_{\nn_i}\big\rangle+\sum^3_{i,j=1}\big\langle\Delta_{-1}(\CP\delta_{\hh_j}),\Delta_{-1}\delta_{\nn_i}\big\rangle\nonumber\\
&\qquad+\sum^3_{i=1}\big\langle\Delta_{-1}\delta_{\FF_i},\Delta_{-1}\delta_{\nn_i}\big\rangle\nonumber\\
&\quad\eqdefa I_5+I_6+I_7,
\end{align}
where we have utilized the short notation similar to (\ref{delta-hh-i}). For the term $I_5$, we can infer from Lemma \ref{commutator} that
\begin{align*}
I_5&=\sum^3_{i=1}\big\langle[\Delta_{-1},\CP]\nabla\delta_{\vv},\Delta_{-1}\delta_{\nn_i}\big\rangle+\sum^3_{i=1}\big\langle\CP\nabla\Delta_{-1}\delta_{\vv},\Delta_{-1}\delta_{\nn_i}\big\rangle\\
&\leq C\big(\|\mathcal{P}\|_{L^{\infty}}+\|\nabla\mathcal{P}\|_{L^4}+\|\nabla^2\mathcal{P}\|_{L^2})\|\delta_{\vv}\|_{B^{-s}_{2,-\infty}}\|\Delta_{-1}\delta_{\Fp}\|_{L^2}\\
&\leq CF(t)\big(\CV(t)+\CU(t)\big).
\end{align*}
Recalling the expression of $\delta_{\hh_i}(i=1,2,3)$ in (\ref{delta-hh-i}), and applying Lemma \ref{useful-inequality}
 and Lemma \ref{commutator}, the term $I_6$ can be handled as
 \begin{align*}
I_6\leq&\sum_{i=1}^3\bigg[\|[\Delta_{-1},\CP^2]\nabla^2\delta_{\nn_i}\|_{L^2}+\|\CP^2 \Delta_{-1}\nabla^2\delta_{\nn_i}\|_{L^2}\\
&+\sum^3_{k=1}\bigg(\|\Delta_{-1}\big(\CP^2\nabla^2\nn_i^{(2)}\delta_{\nn_k}\big)\|_{L^2}+\|\Delta_{-1}\big(\CP^2\nabla\nn^{(1)}_k\nabla\delta_{\nn_i}\big)\|_{L^2}\\
&+\|\Delta_{-1}\big(\CP^2\nabla\nn^{(2)}_i\nabla\delta_{\nn_k}\big)\|_{L^2}+\sum^2_{\alpha=1}\|\Delta_{-1}\big(\CP^2\nabla\nn_i^{(\alpha)}\nabla\delta_{\nn_k}\big)\|_{L^2}\nonumber\\
&+\sum^2_{\alpha=1}\|\Delta_{-1}\big(\CP^2\nabla\nn^{(2)}_i\nabla\nn^{(\alpha)}_k\delta_{\nn_k}\big)\|_{L^2}+\|\Delta_{-1}\big(\CP^2\nabla\nn^{(1)}_k\nabla\nn^{(1)}_k\delta_{\nn_i}\big)\|_{L^2}\bigg)\bigg]\|\Delta_{-1}\delta_{\nn_i}\|_{L^2}\nonumber\\
\leq&C\CU^{\frac{1}{2}}(t)F^{\frac{1}{2}}(t)\|\Delta_{-1}\delta_{\Fp}\|_{L^2}\\
&+C\Big(\|\nabla\CP^2\|_{L^4}+\|(\nabla\Fp^{(1)},\nabla\Fp^{(1)})\|_{L^4}\Big)\|\delta_{\Fp}\|_{B^{1-s}_{2,\infty}}^{\frac{1}{2}}\sum_{|j'+1|\leq9}2^{\frac{j'}{2}}\|\Delta_{j'}\nabla\delta_{\Fp}\|_{L^2}^{\frac{1}{2}}\|\Delta_{-1}\delta_{\Fp}\|_{L^2}\\
\leq& CF(t)\CU(t).
 \end{align*}
It remains to deal with the term $I_7$. Using the definitions of $\delta_{\FF_i}(i=1,2,3)$, we have
 \begin{align*}
I_7=&-\sum^3_{i=1}\big\langle\Delta_{-1}(\vv^{(1)}\cdot\nabla\delta_{\nn_i}),\Delta_{-1}\delta_{\nn_i}\big\rangle-\sum^3_{i=1}\big\langle\Delta_{-1}(\delta_{\vv}\cdot\nn^{(2)}_i),\Delta_{-1}\delta_{\nn_i}\big\rangle\\
&+\sum_{i=1}^{3}\big\langle\Delta_{-1}(\mathcal{P}\nabla\vv^{(2)}\delta_{\nn_i}),\Delta_{-1}\delta_{\nn_i}\big\rangle+\sum_{i,j=1}^{3}\big\langle\Delta_{-1}(\mathcal{P}\delta_{\nn_i}\hh^{(2)}_j),\Delta_{-1}\delta_{\nn_i}\big\rangle\\
\eqdefa&I_{71}+I_{72}+I_{73}+I_{74}.
 \end{align*}
Armed with Lemma \ref{useful-inequality}, we can derive that
\begin{align*}
&I_{71}\leq C\big(\|\vv^{(1)}\|_{H^1}+\|\vv^{(1)}\|_{L^4}\big)\|\delta_{\Fp}\|_{B^{1-s}_{2,\infty}}\|\Delta_{-1}\delta_{\Fp}\|_{L^2}\leq CF(t)\CU(t),\\
&I_{72}\leq C\big(\|\nabla\Fp^{(2)}\|_{H^1}+\|\nabla\Fp^{(2)}\|_{L^4}\big)\|\delta_{\vv}\|_{B^{-s}_{2,\infty}}\|\Delta_{-1}\delta_{\Fp}\|_{L^2}\leq CF(t)\big(\CV(t)+\CU(t)\big),\\
&I_{73}\leq C\big(\|\mathcal{P}\|_{L^{\infty}}+\|\nabla\mathcal{P}\|_{L^2})\|\nabla\vv^{(2)}\|_{L^2}\|\delta_{\Fp}\|_{B^{1-s}_{2,\infty}}\|\Delta_{-1}\delta_{\Fp}\|_{L^2}\leq CF(t)\CU(t).
\end{align*}
Similar to the estimate of the term $I_6$, from Lemma \ref{useful-inequality} and Lemma \ref{commutator}, we also can infer that
\begin{align*}
I_{74}\leq CF(t)\CU(t).
\end{align*}
Therefore, combining the above estimates for $I_i(i=5,6,7)$ with (\ref{Delta--1Fp-L2}), the above Claim 1 can be completed.

Next, we give the second claim about the estimate of $\delta_{\Fp}$.

{\bf Claim 2}: {\it There exists a constant $C>0$ such that for any $j\geq -1$ and $\ve\geq 0$, it follows that}
\begin{align}\label{wj-L2-ud-dt}
&\frac{\ud}{\ud t}\int_{\mathbb{R}^2}W^j(\xx,t)\ud\xx+\sum_{k=1}^{3}\frac{1}{\chi_k}\|\mathcal{H}^{\Delta_j}_k\|_{L^2}^2\nonumber\\
&\leq-\Big\langle\frac{1}{2}\Delta_j\delta_{\BOm}\cdot\aaa^{(1)}_1,\mathcal{H}^{\Delta_j}_3\Big\rangle-\Big\langle\frac{1}{2}\Delta_j\delta_{\BOm}\cdot\aaa^{(1)}_2,\mathcal{H}^{\Delta_j}_2\Big\rangle-\Big\langle\frac{1}{2}\Delta_j\delta_{\BOm}\cdot\aaa^{(1)}_3,\mathcal{H}^{\Delta_j}_1\Big\rangle\nonumber\\
&\quad+\Big\langle\frac{\eta_3}{\chi_3}\Delta_j\delta_{\A}\cdot\sss^{(1)}_3,\mathcal{H}^{\Delta_j}_3\Big\rangle+\Big\langle\frac{\eta_2}{\chi_2}\Delta_j\delta_{\A}\cdot\sss^{(1)}_4,\mathcal{H}^{\Delta_j}_2\Big\rangle+\Big\langle\frac{\eta_1}{\chi_1}\Delta_j\delta_{\A}\cdot\sss^{(1)}_5,\mathcal{H}^{\Delta_j}_1\Big\rangle\nonumber\\
&\quad+C2^{2js}F(t)\Phi(t)+\ve\sum_{l=j-4}^{j+4}2^{2l}\|\Delta_{l}\delta_{\vv}\|_{L^2}^2+\ve\sum_{l=j-9}^{j+9}2^{4l}\|\Delta_l\delta_{\Fp}\|_{L^2}^2,
\end{align}
{\it where $W^j(\xx,t)$ is defined by \eqref{Wj-xt}, and $\CH^{\Delta_j}_k(k=1,2,3)$ are defined by \eqref{H-Delta-j}, respectively}.

Before proving Claim 2, we calculate some derivatives of $W^j(\xx,t)$ with respect to the frame. For convenience, we denote
\begin{align*}
&\CE[\Fp^{(1)},\Delta_j\delta_{\Fp}](t)\eqdefa\int_{\mathbb{R}^2}W^j(\xx,t)\ud\xx,\\
&\CG\eqdefa-\frac{\delta\CE}{\delta(\Delta_j\delta_{\Fp})},\quad \CG_i=-\frac{\delta\CE}{\delta(\Delta_j\delta_{\nn_i})},\quad \CG=(\CG_1,\CG_2,\CG_3).
\end{align*}
By a direct calculation,  we obtain
\begin{align}
\CG_i&=\gamma_i\Delta\Delta_j\delta_{\nn_i}+k_i\nabla\rm{div}\Delta_j\delta_{\nn_i}-\sum_{\alpha=1}^{3}k_{i\alpha}\nabla\times(\nabla\times\Delta_j\delta_{\nn_i}\cdot\nn^{(1)}_\alpha\otimes\nn^{(1)}_\alpha),\label{CG-i}\\
\frac{\delta \CE}{\delta\nn^{(1)}_i}&=\sum_{\alpha=1}^{3}k_{\alpha i}(\nn^{(1)}_i\cdot\nabla\times\Delta_j\delta_{\nn_{\alpha}})(\nabla\times\Delta_j\delta_{\nn_{\alpha}}).\label{delta-CE-Fp1}
\end{align}
Then, from (\ref{CG-i}) and Lemma \ref{b-inequalities} we get
\begin{align}\label{CG-i-L2}
\|\CG_i\|_{L^2}\leq C\Big(\|\nabla^2\Delta_j\delta_{\nn_i}\|_{L^2}+\sum_{\alpha=1}^{3}\||\nabla\nn^{(1)}_{\alpha}|\nabla\Delta_j\delta_{\nn_i}\|_{L^2}\Big)\leq C2^{2j}\|\Delta_j\delta_{\nn_i}\|_{L^2},
\end{align}
Armed with the expressions of $\delta_{\hh_i}(i=1,2,3)$ in (\ref{delta-hh-i}),  $\Delta_j\delta_{\hh_i}$ can be expressed in short notation,
\begin{align}\label{Delta-j-delta-hh-i}
\Delta_j\delta_{\hh_i}=&\,\CG_i+\sum_{\alpha=1}^3k_{i\alpha}\bigg((\nabla\nn_{\alpha}^{(1)}\cdot\nabla\times\Delta_j\delta_{\nn_i})\times\nn_{\alpha}^{(1)}+(\nn_{\alpha}^{(1)}\cdot\nabla\times\Delta_j\delta_{\nn_i})\nabla\times\nn_{\alpha}^{(1)}\bigg)\nonumber\\
&+\sum^3_{l=1}\Delta_j\big(\CP\nabla^2\nn^{(2)}_i\delta_{\nn_i}\big)+\sum^3_{l=1}\Delta_j\big(\CP\nabla\nn^{(1)}_l\nabla\delta_{\nn_i}\big)+\sum^3_{l=1}\Delta_j\big(\CP\nabla\nn^{(2)}_i\nabla\delta_{\nn_l}\big)\nonumber\\
&+\sum^3_{l=1}\sum^2_{\alpha=1}\Delta_j\big(\CP\nabla\nn^{(\alpha)}_i\nabla\delta_{\nn_l}\big)+\sum^3_{l=1}\sum^2_{\alpha=1}\Delta_j\big(\CP\nabla\nn^{(2)}_i\nabla\nn^{(\alpha)}_l\delta_{\nn_l}\big)\nonumber\\
&+\sum^3_{l=1}\Delta_j\big(\CP\nabla\nn^{(1)}_l\nabla\nn^{(1)}_l\delta_{\nn_i}\big)-[\Delta_j,\CP]\nabla(\nabla\times\delta_{\nn_i})\nonumber\\
\eqdefa&\,\CG_i+\CT_i.
\end{align}
Furthermore, from (\ref{Delta-j-delta-hh-i}), together with Lemma \ref{useful-inequality}, Lemma \ref{commutator} and
the Sobolev's embedding $H^{1}(\BR)\hookrightarrow L^{4}(\BR)$, we deduce that
\begin{align}\label{CT-i-L2}
\|\CT_i\|_{L^2}\leq&C\|\nabla\Fp^{(1)}\|_{L^{4}}\|\nabla\Delta_j\delta_{\nn_i}\|_{L^4}+C2^{js}\|\delta_{\nn_i}\|_{B^{1-s}_{2,\infty}}\Big(\|\nabla^2\nn_{i}^{(2)}\|_{L^2}+\|(\nabla\Fp^{(1)},\nabla\Fp^{(2)})\|_{L^4}^2\Big)\nonumber\\
&+C2^{js}\|\delta_{\nn_i}\|_{B^{1-s}_{2,\infty}}\Big(\|(\nabla\Fp^{(1)},\nabla\Fp^{(2)})\|_{H^1}+\|\nabla\CP\|_{L^4}\|(\nabla\Fp^{(1)},\nabla\Fp^{(2)})\|_{L^4}+\|\nabla^2\CP\|_{L^2}\Big)\nonumber\\
&+C2^{\frac{js}{2}}\|(\nabla\Fp^{(1)},\nabla\Fp^{(2)})\|_{L^4}\|\delta_{\nn_i}\|_{B^{1-s}_{2,\infty}}^{\frac{1}{2}}\sum_{j'=j-9}^{j+9}2^{\frac{j'}{2}}\|\Delta_{j'}\nabla\delta_{\nn_i}\|_{L^2}^{\frac{1}{2}}\nonumber\\
\leq&C2^{js}F^{\frac{1}{2}}(t)\CU^{\frac{1}{2}}(t)+C2^{\frac{js}{2}}F^{\frac{1}{4}}(t)\CU^{\frac{1}{4}}(t)\sum_{j'=j-9}^{j+9}2^{\frac{j'}{2}}\|\Delta_{j'}\nabla\delta_{\nn_i}\|_{L^2}^{\frac{1}{2}}.
\end{align}

We now return to the proof of Claim 2. We denote
\begin{align*}
\delta_{\hh}=(\delta_{\hh_1},\delta_{\hh_2},\delta_{\hh_3}),\quad \Delta_j\delta_{\hh}=(\Delta_j\delta_{\hh_1},\Delta_j\delta_{\hh_2},\Delta_j\delta_{\hh_3}),\quad \CT=(\CT_1,\CT_2,\CT_3).
\end{align*}
Then, by the above definition of $\CE$, we obtain
\begin{align}\label{CE-nenergy-dt}
&\frac{\ud}{\ud t}\int_{\mathbb{R}^2}W^j(\xx,t)\ud\xx=\int_{\mathbb{R}^2}\Big(\frac{\delta\CE}{\delta(\Delta_j\delta_{\Fp})}\cdot\frac{\partial\Delta_j\delta_{\Fp}}{\partial t}+\frac{\delta\CE}{\delta\Fp^{(1)}}\cdot\frac{\partial\Fp^{(1)}}{\partial t}\Big)\ud\xx\nonumber\\
&\quad=\int_{\mathbb{R}^2}\sum_{i=1}^{3}\frac{\partial\Delta_j\delta_{\nn_i}}{\partial t}\cdot(\CT_i-\Delta_j\delta_{\hh_i})\ud\xx+\int_{\mathbb{R}^2}\sum_{i=1}^{3}\frac{\delta \CE}{\delta\nn^{(1)}_i}\cdot\frac{\partial\nn^{(1)}_i}{\partial t}\ud\xx\nonumber\\
&\quad\eqdefa \CA_1+\CA_2.
\end{align}
First, using (\ref{delta-CE-Fp1}) and Lemma \ref{b-inequalities}, the term $\CA_2$ can be estimated as
\begin{align*}
\CA_2&\leq C\sum_{i,k=1}^{3}2^{3j}\|\Delta_j\delta_{\nn_i}\|^2_{L^2}\|\partial_t\nn_{k}^{(1)}\|_{L^2}\\
&\leq\ve\sum_{i=1}^{3}2^{4j}\|\Delta_j\delta_{\nn_i}\|_{L^2}^2+C2^{2j}\sum_{i,k=1}^{3}\|\partial_t\nn_{k}^{(1)}\|_{L^2}^2\|\Delta_j\delta_{\nn_i}\|_{L^2}^2\\
&\leq\ve2^{4j}\|\Delta_j\delta_{\Fp}\|_{L^2}^2+C2^{2js}F(t)\CU(t).
\end{align*}
It remains to deal with the term $\CA_1$. Taking advantage of the equations (\ref{n1-uniquence-equ})-(\ref{n3-uniquence-equ}),
we deduce that
\begin{align}\label{CA-1}
\CA_1&=-\Big\langle\frac{1}{2}\Delta_j\delta_{\BOm}\cdot\aaa^{(1)}_1,\mathcal{H}^{\Delta_j}_3\Big\rangle-\Big\langle\frac{1}{2}\Delta_j\delta_{\BOm}\cdot\aaa^{(1)}_2,\mathcal{H}^{\Delta_j}_2\Big\rangle-\Big\langle\frac{1}{2}\Delta_j\delta_{\BOm}\cdot\aaa^{(1)}_3,\mathcal{H}^{\Delta_j}_1\Big\rangle\nonumber\\
&\quad+\Big\langle\frac{\eta_3}{\chi_3}\Delta_j\delta_{\A}\cdot\sss^{(1)}_3,\mathcal{H}^{\Delta_j}_3\Big\rangle+\Big\langle\frac{\eta_2}{\chi_2}\Delta_j\delta_{\A}\cdot\sss^{(1)}_4,\mathcal{H}^{\Delta_j}_2\Big\rangle+\Big\langle\frac{\eta_1}{\chi_1}\Delta_j\delta_{\A}\cdot\sss^{(1)}_5,\mathcal{H}^{\Delta_j}_1\Big\rangle\nonumber\\
&\quad-\sum^3_{k=1}\frac{1}{\chi_k}\|\mathcal{H}^{\Delta_j}_k\|_{L^2}^2+\underbrace{\sum^3_{i=1}\big\langle\CT_i,\partial_t\Delta_j\delta_{\nn_i}\big\rangle}_{\CA_{11}}\underbrace{-\sum^3_{i=1}\big\langle\Delta_j\delta_{\hh_i},\Delta_j\delta_{\FF_i}\big\rangle}_{\CA_{12}}\nonumber\\
&\quad\underbrace{-\sum^3_{i=1}\big\langle[\Delta_j,\CP]\nabla\delta_{\vv}+[\Delta_j,\CP]\delta_{\hh_i},\Delta_j\delta_{\hh_i}\big\rangle}_{\CA_{13}}.
\end{align}
Using Lemma \ref{useful-inequality} and Lemma \ref{commutator} the term $\CA_{11}$ can be estimated as follows:
\begin{align*}
\CA_{11}=&\,\Big\langle\CT_1\cdot\nn^{(1)}_2-\CT_2\cdot\nn^{(1)}_1,\frac{1}{2}\Delta_j\delta_{\BOm}\cdot\aaa^{(1)}_1+\frac{\eta_3}{\chi_3}\Delta_j\delta_{\A}\cdot\sss^{(1)}_3\Big\rangle\nonumber\\
&+\Big\langle\CT_3\cdot\nn^{(1)}_1-\CT_1\cdot\nn^{(1)}_3,\frac{1}{2}\Delta_j\delta_{\BOm}\cdot\aaa^{(1)}_2+\frac{\eta_2}{\chi_2}\Delta_j\delta_{\A}\cdot\sss^{(1)}_4\Big\rangle\nonumber\\
&+\Big\langle\CT_2\cdot\nn^{(1)}_3-\CT_3\cdot\nn^{(1)}_2,\frac{1}{2}\Delta_j\delta_{\BOm}\cdot\aaa^{(1)}_3+\frac{\eta_1}{\chi_1}\Delta_j\delta_{\A}\cdot\sss^{(1)}_5\Big\rangle\nonumber\\
&-\sum_{k=1}^3\Big\langle\CT\cdot V^{(1)}_k,\frac{1}{\chi_k}\CH^{\Delta_j}_k\Big\rangle+\sum_{i=1}^3\big\langle\CT_i,\Delta_j\delta_{\FF_i}\big\rangle\nonumber\\
&+\sum_{i=1}^3\big\langle\CT_i,[\Delta_j,\CP]\nabla\delta_{\vv}+[\Delta_j,\CP]\delta_{\hh_i}\big\rangle\nonumber\\
\leq&\,C\sum_{i=1}^3\bigg(\|\Delta_j\nabla\delta_{\vv}\|_{L^2}+\|\Delta_j(\vv^{(1)}\cdot\nabla\delta_{\nn_i})\|_{L^2}+\|\Delta_j(\delta_{\vv}\cdot\nn_i^{(2)})\|_{L^2}\nonumber\\
&+\sum_{k=1}^3\Big(\|\Delta_j\delta_{\hh_k}\|_{L^2}+\|\Delta_j(\CP\nabla\vv^{(2)}\delta_{\nn_k})\|_{L^2}+\sum_{l=1}^3\|\Delta_j(\CP\hh_k^{(2)}\delta_{\nn_l})\|_{L^2}\Big)\nonumber\\
&+\|[\Delta_j,\CP]\nabla\delta_{\vv}\|_{L^2}+\|[\Delta_j,\CP]\delta_{\hh_i}\|_{L^2}\bigg)\|\CT_i\|_{L^2}\nonumber\\
\leq&\,C\sum_{i=1}^3\bigg(\|\Delta_j\nabla\delta_{\vv}\|_{L^2}+2^{2j}\|\Delta_j\delta_{\nn_i}\|_{L^2}+2^{js}\big(\CU^{\frac{1}{2}}(t)+\CV^{\frac{1}{2}}(t)\big)F^{\frac{1}{2}}(t)\nonumber\\
&+2^{\frac{js}{2}}F^{\frac{1}{4}}(t)\CU^{\frac{1}{4}}(t)\sum_{|j-j'|\leq9}2^{\frac{j'}{2}}\|\Delta_{j'}\nabla\delta_{\nn_i}\|_{L^2}^{\frac{1}{2}}\nonumber\\
&+2^{\frac{js}{2}}F^{\frac{1}{4}}(t)\CV^{\frac{1}{4}}(t)\sum_{|j-j'|\leq4}2^{\frac{j'}{2}}\|\Delta_{j'}\delta_{\vv}\|_{L^2}^{\frac{1}{2}}\bigg)\|\CT_i\|_{L^2}\nonumber\\
\leq&\,\ve\sum_{|j-j'|\leq9}2^{4j'}\|\Delta_{j'}\delta_{\Fp}\|_{L^2}^2+\ve\sum_{|j-j'|\leq4}2^{2j'}\|\Delta_{j'}\delta_{\vv}\|_{L^2}^2\nonumber\\
&+C2^{2js}\big(\CU(t)+\CV(t)\big)F(t),
\end{align*}
where we have used the expression $V^{(1)}_k(k=1,2,3)$ that are given by
\begin{align*}
V^{(1)}_1=(0,\nn^{(1)}_3,-\nn^{(1)}_2),\quad V^{(1)}_2=(-\nn^{(1)}_3,0,\nn^{(1)}_1),\quad V^{(1)}_3=(\nn^{(1)}_2,-\nn^{(1)}_1,0).
\end{align*}
Similarly, the terms $\CA_{12}$ and $\CA_{13}$ can be handled as, respectively,
\begin{align*}
|\CA_{12}|\leq&\sum_{i=1}^3\|\Delta_j\delta_{\hh_i}\|_{L^2}\|\Delta_j\delta_{\FF_i}\|_{L^2}\nonumber\\
\leq&\sum_{i=1}^3\|\Delta_j\delta_{\hh_i}\|_{L^2}\bigg(\|\Delta_j(\CP\nabla\vv^{(2)}\delta_{\nn_k})\|_{L^2}+\sum_{l=1}^3\|\Delta_j(\CP\hh_k^{(2)}\delta_{\nn_l})\|_{L^2}\nonumber\\
&+\|\Delta_j(\vv^{(1)}\cdot\nabla\delta_{\nn_i})\|_{L^2}+\|\Delta_j(\delta_{\vv}\cdot\nn_i^{(2)})\|_{L^2}\bigg)\nonumber\\
\leq&\,\ve\sum_{|j-j'|\leq9}2^{4j'}\|\Delta_{j'}\delta_{\Fp}\|_{L^2}^2+\ve\sum_{|j-j'|\leq4}2^{2j'}\|\Delta_{j'}\delta_{\vv}\|_{L^2}^2\nonumber\\
&+C2^{2js}\big(\CU(t)+\CV(t)\big)F(t),\\
|\CA_{13}|\leq&\sum_{i=1}^3\big(\|[\Delta_j,\CP]\nabla\delta_{\vv}\|_{L^2}+\|[\Delta_j,\CP]\delta_{\hh_i}\|_{L^2}\big)\|\Delta_j\delta_{\hh_i}\|_{L^2}\nonumber\\
\leq&\,\ve\sum_{|j-j'|\leq9}2^{4j'}\|\Delta_{j'}\delta_{\Fp}\|_{L^2}^2+\ve\sum_{|j-j'|\leq4}2^{2j'}\|\Delta_{j'}\delta_{\vv}\|_{L^2}^2\nonumber\\
&+C2^{2js}\big(\CU(t)+\CV(t)\big)F(t).
\end{align*}
Combining the above estimates of $\CA_{1k}(k=1,2,3)$ with (\ref{CA-1})  we obtain
\begin{align*}
\CA_1&\leq-\Big\langle\frac{1}{2}\Delta_j\delta_{\BOm}\cdot\aaa^{(1)}_1,\mathcal{H}^{\Delta_j}_3\Big\rangle-\Big\langle\frac{1}{2}\Delta_j\delta_{\BOm}\cdot\aaa^{(1)}_2,\mathcal{H}^{\Delta_j}_2\Big\rangle-\Big\langle\frac{1}{2}\Delta_j\delta_{\BOm}\cdot\aaa^{(1)}_3,\mathcal{H}^{\Delta_j}_1\Big\rangle\nonumber\\
&\quad+\Big\langle\frac{\eta_3}{\chi_3}\Delta_j\delta_{\A}\cdot\sss^{(1)}_3,\mathcal{H}^{\Delta_j}_3\Big\rangle+\Big\langle\frac{\eta_2}{\chi_2}\Delta_j\delta_{\A}\cdot\sss^{(1)}_4,\mathcal{H}^{\Delta_j}_2\Big\rangle+\Big\langle\frac{\eta_1}{\chi_1}\Delta_j\delta_{\A}\cdot\sss^{(1)}_5,\mathcal{H}^{\Delta_j}_1\Big\rangle\nonumber\\
&\quad-\sum^3_{k=1}\frac{1}{\chi_k}\|\mathcal{H}^{\Delta_j}_k\|_{L^2}^2+C2^{2js}\big(\CU(t)+\CV(t)\big)F(t)\nonumber\\
&\quad+\ve\sum_{|j-j'|\leq9}2^{4j'}\|\Delta_{j'}\delta_{\Fp}\|_{L^2}^2+\ve\sum_{|j-j'|\leq4}2^{2j'}\|\Delta_{j'}\delta_{\vv}\|_{L^2}^2.
\end{align*}
Hence, plugging the estimates of $\CA_1$ and $\CA_2$ into (\ref{CE-nenergy-dt}) yields (\ref{wj-L2-ud-dt}). This completes the proof of Claim 2.

In the end, to control the higher order derivative term $\|\Delta\Delta_j\delta_{\Fp}\|_{L^2}$ that enable us to close the energy estimate, we need a higher order dissipated estimate for $\sum^3_{k=1}\frac{1}{\chi_k}\|\mathcal{H}^{\Delta_j}_k\|_{L^2}^2$. To be specific, we have the following key lemma.

\begin{lemma}\label{key-lemma}
There exists a constant $C>0$ such that for any $\ve>0$ sufficiently small, it follows that
\begin{align*}
-\sum^3_{k=1}\frac{1}{\chi_k}\|\mathcal{H}^{\Delta_j}_k\|_{L^2}^2 \leq-\frac{2\gamma^2}{\chi}\|\Delta\Delta_j\delta_{\Fp}\|_{L^2}^2+CF(t)\CU(t)+\ve\sum_{l=j-9}^{j+9}2^{4l}\|\Delta_l\delta_{\Fp}\|_{L^2}^2,
\end{align*}
where $\chi=\max\{\chi_1,\chi_2,\chi_3\}$ and $\gamma=\min\{\gamma_1,\gamma_2,\gamma_3\}$, and $\CH^{\Delta_j}_k(k=1,2,3)$ are defined by \eqref{H-Delta-j}.
\end{lemma}

\begin{proof}
First of all, we recall the decomposition (\ref{Orth-decomp-SO3}) formed by the tangential spaces $T_{\Fp^{(\alpha)}}SO(3)(\alpha=1,2)$ and its orthogonal complements.
In (\ref{Orth-decomp-SO3}), taking $B=\Delta\Delta_j\Fp^{(\alpha)}$, we obtain
\begin{align}\label{inner-product-decomposition}
A\cdot\Delta\Delta_j\Fp^{(\alpha)}=&\sum_{k=1}^3\frac{1}{|V^{(\alpha)}_k|^2}(A\cdot V^{(\alpha)}_k)(\Delta\Delta_j\Fp^{(\alpha)}\cdot V^{(\alpha)}_k)\nonumber\\
&+\sum_{k=1}^6\frac{1}{|W^{(\alpha)}_k|^2}(A\cdot W^{(\alpha)}_k)(\Delta\Delta_j\Fp^{(\alpha)}\cdot W^{(\alpha)}_k),\quad \alpha=1,2,
\end{align}
where $V^{(\alpha)}_k(k=1,2,3)$ and $W^{(\alpha)}_k(k=1,\cdots,6)$ are the orthogonal basis of the tangential space $T_{\Fp^{(\alpha)}}SO(3)$ and its associated orthogonal complement space, respectively, that is,
\begin{align*}
V^{(\alpha)}_1=&(0,\nn^{(\alpha)}_3,-\nn^{(\alpha)}_2),\quad V^{(\alpha)}_2=(-\nn^{(\alpha)}_3,0,\nn^{(\alpha)}_1),\quad V^{(\alpha)}_3=(\nn^{(\alpha)}_2,-\nn^{(\alpha)}_1,0),\\
W^{(\alpha)}_1=&(0,\nn^{(\alpha)}_3,\nn^{(\alpha)}_2),\quad W^{(\alpha)}_2=(\nn^{(\alpha)}_3,0,\nn^{(\alpha)}_1),\quad W^{(\alpha)}_3=(\nn^{(\alpha)}_2,\nn^{(\alpha)}_1,0),\\
W^{(\alpha)}_4=&(\nn^{(\alpha)}_1,0,0),\quad W^{(\alpha)}_5=(0,\nn^{(\alpha)}_2,0),\quad W^{(\alpha)}_6=(0,0,\nn^{(\alpha)}_3).
\end{align*}
Armed with the definitions of $W^{(\alpha)}_k(k=1,\cdots,6)$, it follows that
\begin{align*}
\Delta\Delta_j\Fp^{(\alpha)}\cdot W^{(\alpha)}_k=&\Delta_j(\Delta\Fp^{(\alpha)}\cdot W^{(\alpha)}_k)-[\Delta_j,W^{(\alpha)}_k\cdot]\Delta\Fp^{(\alpha)}\\
=&\Delta_j\big(\nabla\cdot(\nabla\Fp^{(\alpha)}\cdot W^{(\alpha)}_k)-\nabla\Fp^{(\alpha)}\cdot \nabla W^{(\alpha)}_k\big)-[\Delta_j,W^{(\alpha)}_k\cdot]\Delta\Fp^{(\alpha)}\\
=&-\Delta_j(\nabla\Fp^{(\alpha)}\cdot \nabla W^{(\alpha)}_k)-[\Delta_j,W^{(\alpha)}_k\cdot]\Delta\Fp^{(\alpha)}.
\end{align*}

Noticing $|V^{(1)}_k|=|V^{(2)}_k|$ and $|W^{(1)}_k|=|W^{(2)}_k|$,  and using (\ref{inner-product-decomposition}) we derive
\begin{align}\label{Delta-decomposition}
A\cdot\Delta\Delta_j\delta_{\Fp}=&\sum^3_{k=1}\frac{1}{|V^{(1)}_k|^2}(A\cdot V^{(1)}_k)(\Delta\Delta_j\delta_{\Fp}\cdot V^{(1)}_k)\nonumber\\
&+\sum^3_{k=1}\frac{1}{|V^{(1)}_k|^2}\Big((A\cdot \delta_{V_k})(\Delta\Delta_j\Fp^{(2)}\cdot V^{(1)}_k)+(A\cdot V^{(2)}_k)(\Delta\Delta_j\Fp^{(2)}\cdot\delta_{V_k})\Big)\nonumber\\
&-\sum^6_{k=1}\frac{1}{|W^{(1)}_k|^2}\bigg\{(A\cdot W^{(1)}_k)\Big(\Delta_j(\nabla\delta_{\Fp}\cdot\nabla W^{(1)}_k)+[\Delta_j,W^{(1)}_k\cdot]\Delta\delta_{\Fp}\Big)\nonumber\\
&+(A\cdot W^{(1)}_k)\Big(\Delta_j(\nabla\Fp^{(2)}\cdot\nabla\delta_{W_k})+[\Delta_j,\delta_{W_k}\cdot]\Delta\Fp^{(2)}\Big)\nonumber\\
&+(A\cdot\delta_{W_k})\Big(\Delta_j(\nabla\Fp^{(2)}\cdot\nabla W^{(2)}_k)+[\Delta_j,W^{(2)}_k\cdot]\Delta\Fp^{(2)}\Big)\bigg\},
\end{align}
where $\delta_{V_k}=V^{(1)}_k-V^{(2)}_k$ and $\delta_{W_k}=W^{(1)}_k-W^{(2)}_k$.

Taking $A=\Delta\Delta_j\delta_{\Fp}$ in (\ref{Delta-decomposition}) and then integrating over $\mathbb{R}^2$, and using Lemma \ref{useful-inequality} and Lemma \ref{commutator}, we can derive
\begin{align*}
\|\Delta\Delta_j\delta_{\Fp}\|_{L^2}^2\leq&\frac{1}{2}\sum_{k=1}^3 \|\Delta\Delta_j\delta_{\Fp}\cdot V^{(1)}_k\|_{L^2}^2+C\sum^3_{k=1}\|\Delta\Delta_j\delta_{\Fp}\|_{L^2}\|\delta_{V_k}\|_{L^{\infty}}\|\Delta\Delta_j\Fp^{(2)}\|_{L^2}\\
&+C\sum^6_{k=1}\|\Delta\Delta_j\delta_{\Fp}\|_{L^2}\Big(\|\Delta_j(\nabla\delta_{\Fp}\cdot\nabla W^{(1)}_k)\|_{L^2}+\|[\Delta_j,W^{(1)}_k\cdot]\Delta\delta_{\Fp}\|_{L^2}\Big)\\
&+C\sum^6_{k=1}\|\Delta\Delta_j\delta_{\Fp}\|_{L^2}\Big(\|\Delta_j(\nabla\Fp^{(2)}\cdot\nabla \delta_{W_k})\|_{L^2}+\|[\Delta_j,\delta_{W_k}\cdot]\Delta\Fp^{(2)}\|_{L^2}\Big)\\
&+C\sum^6_{k=1}\|\Delta\Delta_j\delta_{\Fp}\|_{L^2}\|\delta_{W_k}\|_{L^{\infty}}\Big(\|\Delta_j(\nabla\Fp^{(2)}\cdot\nabla W^{(2)}_k)\|_{L^2}+\|[\Delta_j,W^{(2)}_k\cdot]\Delta\Fp^{(2)}\|_{L^2}\Big)\\
\leq&
\frac{1}{2}\sum_{k=1}^3 \|\Delta\Delta_j\delta_{\Fp}\cdot V^{(1)}_k\|_{L^2}^2+C2^{2js}F(t)\CU(t)+\ve\sum^{j+9}_{l=j-9}2^{4l}\|\Delta_l\delta_{\Fp}\|_{L^2}^2.
\end{align*}
Similarly, by taking  $A=\CG-\gamma\Delta\Delta_j\delta_{\Fp}$ in (\ref{Delta-decomposition}), it follows that
\begin{align*}
\int_{\mathbb{R}^2}(\CG-\gamma\Delta\Delta_j\delta_{\Fp})\cdot\Delta\Delta_j\delta_{\Fp}\ud\xx
\leq&\frac{1}{2}\int_{\BR}\big[(\CG-\gamma\Delta\Delta_j\delta_{\Fp})\cdot V^{(1)}_k\big](\Delta\Delta_j\delta_{\Fp}\cdot V^{(1)}_k)\ud\xx\nonumber\\
&+C2^{2js}F(t)\CU(t)+\ve\sum^{j+9}_{l=j-9}2^{4l}\|\Delta_l\delta_{\Fp}\|_{L^2}^2.
\end{align*}
Furthermore, applying integration by parts, there holds
\begin{align*}
&\int_{\mathbb{R}^2}(\CG-\gamma\Delta\Delta_j\delta_{\Fp})\cdot\Delta\Delta_j\delta_{\Fp}\ud\xx\nonumber\\
&\quad\geq\sum_{i=1}^3\Big(k_i\|\nabla\rm{div}\Delta_j\delta_{\nn_i}\|_{L^2}^2+\sum_{\alpha=1}^3k_{\alpha i}\|\nabla(\nabla\times\Delta_j\delta_{\nn_i}\cdot\nn^{(1)}_{\alpha})\|_{L^2}^2\Big)\nonumber\\
&\qquad-C2^{2js}F(t)\CU(t)-\ve\sum^{j+9}_{l=j-9}2^{4l}\|\Delta_l\delta_{\Fp}\|_{L^2}^2.
\end{align*}
On the other hand, by the definitions of $\CG$ and $\CT$, and using (\ref{CG-i-L2}) and (\ref{CT-i-L2}), we obtain
\begin{align*}
\Big|\int_{\BR}\sum_{k=1}^3 (\CG\cdot V^{(1)}_k)(\CT\cdot V^{(1)}_k)\ud\xx\Big|\leq&C\sum^3_{i=1}\|\CG_i\|_{L^2}\|\CT_i\|_{L^2}\\
\leq&C2^{js}F(t)\CU(t)+\ve\sum_{l=j-9}^{j+9}2^{4l}\|\Delta_l\delta_{\Fp}\|_{L^2}^2.
\end{align*}
Then, collecting the above estimates, and recalling the definitions of $\CH^{\Delta_j}_k(k=1,2,3)$ in (\ref{H-Delta-j}), i.e., $\CH^{\Delta_j}_k=\Delta_j\delta_{\hh}\cdot V^{(1)}_k$,  together with $\Delta_j\delta_{\hh}=\CG+\CT$, we can derive that
\begin{align*}
&\sum^3_{k=1}\frac{1}{\chi_k}\|\mathcal{H}^{\Delta_j}_k\|_{L^2}^2\geq\frac{1}{\chi}\sum^3_{k=1}\|\Delta_j\delta_{\hh}\cdot V^{(1)}_k\|^2_{L^2}\\
&\quad\geq\frac{1}{\chi}\sum_{k=1}^3\|\CG\cdot V^{(1)}_k\|_{L^2}^2+\frac{2}{\chi}\sum_{k=1}^3\int_{\BR}\big(\CG\cdot V^{(1)}_k)(\CT\cdot V^{(1)}_k)\ud\xx\\
&\quad\geq\frac{2\gamma}{\chi}\sum_{k=1}^3\int_{\BR}\big[(\CG-\gamma\Delta\Delta_j\delta_{\Fp})\cdot V^{(1)}_k\big](\Delta\Delta_j\delta_{\Fp}\cdot V^{(1)}_k)\ud\xx+\frac{\gamma^2}{\chi}\|\Delta\Delta_j\delta_{\Fp}\cdot V^{(1)}_k\|_{L^2}^2\\
&\qquad-C2^{js}F(t)\CU(t)-\ve\sum_{l=j-9}^{j+9}2^{4l}\|\Delta_l\delta_{\Fp}\|_{L^2}^2\\
&\quad\geq\frac{4\gamma}{\chi}\int_{\BR}(\CG-\gamma\Delta\Delta_j\delta_{\Fp})\cdot\Delta\Delta_j\delta_{\Fp}\ud\xx+\frac{2\gamma^2}{\chi}\|\Delta\Delta_j\delta_{\Fp}\|_{L^2}^2\\
&\qquad-C2^{js}F(t)\CU(t)-\ve\sum_{l=j-9}^{j+9}2^{4l}\|\Delta_l\delta_{\Fp}\|_{L^2}^2\\
&\quad\geq\frac{4\gamma}{\chi}\sum_{i=1}^3\Big(k_i\|\nabla\rm{div}\Delta_j\delta_{\nn_i}\|_{L^2}^2+\sum_{\alpha=1}^3k_{\alpha i}\|\nabla\times(\nabla\times\Delta_j\delta_{\nn_i}\cdot\nn^{(1)}_{\alpha})\|_{L^2}^2\Big)\\
&\qquad+\frac{2\gamma^2}{\chi}\|\Delta\Delta_j\delta_{\Fp}\|_{L^2}^2-C2^{js}F(t)\CU(t)-\ve\sum_{l=j-9}^{j+9}2^{4l}\|\Delta_l\delta_{\Fp}\|_{L^2}^2.
\end{align*}
Hence, we complete the proof the key lemma.
\end{proof}

\subsection{Proof of Theorem \ref{uniqueness-theorem}}

We are now in a position to complete the proof Theorem \ref{uniqueness-theorem}. First of all, we know from Theorem \ref{global-posedness-theorem} in subsection 1.3 and Proposition 4.2 in \cite{LWX} that, for any $\theta,M>0$, the pair $(\Fp,\vv)$ satisfies
\begin{align*}
(\Fp,\vv)\in& V(0,T_1-\theta)\times H(0,T_1-\theta)\cup\cdots\cup V(T_{L-1},T_{L}-\theta)\times H(T_{L-1},T_{L}-\theta)\\
&\cup V(T_L,M)\times H(T_L,M),
\end{align*}
where $\{T_l\}^L_{l=1}$ are the finite number of singular times.
Assume that $T^{(i)}_1$ is the first blow-up time of $(\Fp^{(i)},\vv^{(i)})(i=1,2)$. By means of Proposition 4.2 in \cite{LWX}, we may find that the following regularity for $(\Fp^{i},\vv^{(i)})$ holds,
\begin{align}\label{regularity-Fp-vv}
\int_{\mathbb{R}^2\times[0,T_1-\theta]}|\nabla^2\Fp^{(i)}|^2+|\nabla\Fp^{(i)}|^4+|\nabla\vv^{(i)}|^2+|\vv^{(i)}|^4\ud\xx\ud t<+\infty,
\end{align}
where $\theta>0$ and $T_1=\min\{T^{(1)}_1,T^{(2)}_1\}$. Then, taking advantage of the equations (\ref{new-frame-equation-n1})--(\ref{new-frame-equation-n3}), we can infer that
\begin{align}\label{partial-t-Fp}
\partial_t\Fp^{(i)}\in L^2(\mathbb{R}^2\times[0,T_1-\theta]).
\end{align}
Further, (\ref{regularity-Fp-vv}) and (\ref{partial-t-Fp}) imply that $F(t)$ is a locally integrable function, i.e., $F(t)\in L^1(0,T_1-\theta)$.

On the other hand, summing up (\ref{differen-delta-vv-L2final}), (\ref{claim-Delta-1-L2}) and (\ref{wj-L2-ud-dt}), and together with Lemma \ref{key-lemma}, it follows that
\begin{align}\label{final-energy-L2}
&\frac{\ud}{\ud t}\Big(\frac{1}{2}\|\Delta_j\delta_{\vv}\|_{L^2}^2+\int_{\mathbb{R}^2}W^j(\xx,t)\ud\xx+\frac{1}{2}\|\Delta_{-1}\delta_{\Fp}\|_{L^2}\Big)\\
&\qquad+ca_j\big(2^{2j}\|\Delta_j\delta_{\vv}\|_{L^2}^2+2^{4j}\|\Delta_j\delta_{\nn_i}\|_{L^2}^2\big)\\
&\quad\leq C2^{2js}F(t)\Phi(t)+\ve\sum^{j+4}_{l=j-4}2^{2l}\|\Delta_{l}\delta_{\vv}\|_{L^2}^2+\ve\sum_{l=j-9}^{j+9}2^{4l}\|\Delta_l\delta_{\Fp}\|_{L^2}^2.
\end{align}
Recalling the relation (\ref{Wj-relation}) and taking the supreme in $j$, and choosing $\ve>0$ sufficiently small, we can derive from (\ref{final-energy-L2}) that
\begin{align}\label{energy-estimate}
\Phi(t)\leq C\int_{0}^{t}F(\tau)\Phi(\tau)\ud\tau.
\end{align}
Then, applying Lemma \ref{Osgood-lemma}, that is the Osgood lemma, to (\ref{energy-estimate}) and using the given same initial data, we find that $\Phi(t)=0$ for any $\theta>0$ and $t\in[0,T_1-\theta]$, where $T_1$ is the first singular time. Consequently, we have  $(\Fp^{(1)},\vv^{(1)})(t)=(\Fp^{(2)},\vv^{(2)})(t)$ for any $t\in[0,T_1)$. Indeed, we can further get $(\Fp^{(1)},\vv^{(1)})(T_1)=(\Fp^{(2)},\vv^{(2)})(T_1)$, since $(\Fp^{(i)},\vv^{(i)})(t)$ is weakly continuous for any $t\in[0, +\infty)$, i.e., $(\Fp^{(i)},\vv^{(i)})\in C_{w}([0,+\infty);H^1_{\Fp^*}\times L^2)$.
The similar argument implies that there exists the second singular time $T_2>0(T_2>T_1)$ and such that $(\vv^{(1)},\Fp^{(1)})(t)=(\vv^{(2)},\Fp^{(2)})(t)$ with $t\in[T_1,T_2)$. We can thus obtain $(\Fp^{(1)},\vv^{(1)})(t)=(\Fp^{(2)},\vv^{(2)})(t)$ for any $t\in[0,+\infty)$, since the number of singular time for weak solution is finite. Hence, we finish the proof of Theorem \ref{uniqueness-theorem}.

\bigskip
\noindent{\bf Acknowledgments.}
 Sirui Li is partially supported by the NSFC under grant No. 12061019 and by the Growth Foundation for Youth Science and Technology Talent of Educational Commission of Guizhou Province of China under grant No. [2021]087. Jie Xu is partially supported by the NSFC under grant Nos. 12288201 and 12001524.

\end{document}